\documentclass[11pt,leqno]{amsart}
\usepackage{amsmath, amscd, amsthm, amssymb, graphics, xypic, mathrsfs, setspace, fancyhdr, bm, pdfsync, enumitem, mathptmx}
\usepackage[usenames, dvipsnames, svgnames, table]{xcolor}
\usepackage[letterpaper,top=1.05in, bottom=1.05in, left=1.05in, right=1.05in]{geometry}
\usepackage[colorlinks=true]{hyperref}

\newcommand{\Spec}{\operatorname{Spec}}

\newcommand{\coker}{\operatorname{coker}}

\newcommand{\longtwoheadrightarrow}{\relbar\joinrel\twoheadrightarrow}

\newcommand{\GW}{\mathbf{GW}}

\newcommand{\Hom}{\operatorname{Hom}}

\newcommand{\real}{{\mathbb R}}

\newcommand{\Z}{{\mathbb Z}}

\newcommand{\A}{{\mathbb A}}

\newcommand{\Osc}{{\mathscr{O}}}
\newcommand{\Esc}{{\mathscr{E}}}

\newcommand{\aone}{{\mathbb A}^1}
\newcommand{\pone}{{\mathbb P}^1}
\newcommand{\GL}{{\mathrm{GL}}}

\newcommand{\gm}[1]{{{\mathbb G}_{m}^{#1}}}
\newcommand{\Gm}{{\gm{}}}
\renewcommand{\Lc}{{\mathcal L}}

\newcommand{\et}{\text{\'et}}

\newcommand{\bpi}{\bm{\pi}}
\newcommand{\piaone}{{\bpi}^{\aone}}

\newcommand{\Nis}{{\operatorname{Nis}}}
\newcommand{\Zar}{\operatorname{Zar}} 

\newcommand{\CHW}{{\widetilde{\mathrm{CH}}}}
\newcommand{\CH}{{\mathrm{CH}}}

\newcommand{\Shv}{{\mathrm{Shv}}}
\newcommand{\Sm}{\mathrm{Sm}}

\newcommand{\Spc}{\mathrm{Spc}}
\newcommand{\Ab}{\mathrm{Ab}}

\newcommand{\K}{{{\mathbf K}}}

\newcommand{\KMW}{\K^{\mathrm{MW}}}
\newcommand{\KM}{\K^{\mathrm M}}
\newcommand{\Pic}{\operatorname{Pic}}

\newcommand{\Hr}{{\mathrm {H}}}

\newcommand{\Ibf}{\mathbf{I}}

\newcommand{\Rb}{\mathbb{R}}

\newcommand{\Kr}{\mathrm{K}}
\newcommand{\GWr}{\mathrm{GW}}

\newcommand{\Abf}{\mathbf{A}}
\newcommand{\Cr}{\mathrm{C}}
\newcommand{\RS}{\mathrm{RS}}

\newcommand{\Zb}{\mathbb{Z}}
\newcommand{\Cb}{\mathbb{C}}
\newcommand{\Bbf}{\mathbf{B}}
\newcommand{\Kbf}{\mathbf{K}}

\newcommand{\Er}{\mathrm{E}}

\newcommand{\Sr}{\mathrm{S}}

\newcommand{\Abb}{\mathbb{A}}

\newcommand{\Vsc}{\mathscr{V}}

\newcommand{\BGL}{\operatorname{BGL}}
\newcommand{\Mr}{\mathrm{M}}
\newcommand{\Hsc}{\mathcal{H}}
\newcommand{\im}{\operatorname{im}}
\newcommand{\Gbf}{\mathbf{G}}

\newcommand{\Um}{\operatorname{Um}}
\newcommand{\Br}{\mathrm{B}}
\newcommand{\diag}{\operatorname{diag}}
\newcommand{\MW}{\mathrm{MW}}

\newcommand{\Map}{\operatorname{Map}}

\newcommand{\Hbf}{\mathbf{H}}
\newcommand{\Qbf}{\mathbf{Q}}

\newcommand{\Tbf}{\mathbf{T}}
\newcommand{\Sbf}{\mathbf{S}}
\newcommand{\Fbf}{\mathbf{F}}

\newcommand{\Dr}{\mathrm{D}}

\renewcommand{\setminus}{\smallsetminus}

\newcounter{intro}
\theoremstyle{plain}
\newtheorem{thm}{Theorem}[section]

\newtheorem{lem}[thm]{Lemma}
\newtheorem{cor}[thm]{Corollary}
\newtheorem{prop}[thm]{Proposition}
\newtheorem*{claim*}{Claim} 

\newtheorem*{question*}{Main question}

\newtheorem*{thm*}{Theorem}
\newtheorem*{problem*}{Problem}

\newtheorem{thmintro}{Theorem}

\theoremstyle{definition}
\newtheorem{defn}[thm]{Definition}

\theoremstyle{remark}
\newtheorem{rem}[thm]{Remark}
\newtheorem{remintro}[thmintro]{Remark}

\newtheorem{ex}[thm]{Example}

\numberwithin{equation}{section}

\usepackage{tikz}
\usetikzlibrary{arrows,matrix,cd}

\usepackage{stmaryrd}

\begin{document}
\pagestyle{fancy}
\renewcommand{\sectionmark}[1]{\markright{\thesection\ #1}}
\fancyhead{}
\fancyhead[LO,R]{\bfseries\footnotesize\thepage}
\fancyhead[LE]{\bfseries\footnotesize\rightmark}
\fancyhead[RO]{\bfseries\footnotesize\rightmark}
\chead[]{}
\cfoot[]{}
\setlength{\headheight}{1cm}

\title[Suslin's cancellation conjecture on varieties with few real points]{{\bf Suslin's cancellation conjecture on real affine varieties with few real points}}
\date{}
\author{Sourjya Banerjee}
\address{Sourjya Banerjee \\
Department of Mathematics \\ 
Indian Institute of Technology Hyderabad \\
 Kandi, Sangareddy, Telangana \\
 502284, India} 
\email{\href{mailto:sourjya@math.iith.ac.in, sourjya91@gmail.com}{sourjya@math.iith.ac.in}}
\email{\href{mailto:sourjya91@gmail.com}{sourjya91@gmail.com}}
\urladdr{\url{https://sites.google.com/view/sourjyabanerjee/home}}
\author{Jean Fasel}
\address{Jean Fasel \\
Universit\'e Grenoble Alpes \\
Institut Fourier \\
CS 40700 \\
38058 Grenoble Cedex 9 \\
France} 
\email{\href{mailto:Jean.Fasel@univ-grenoble-alpes.fr}{Jean.Fasel@univ-grenoble-alpes.fr}}
\urladdr{\url{https://www-fourier.univ-grenoble-alpes.fr/~faselj/index.html}}
\author{Samuel Lerbet}
\address{Samuel Lerbet \\
DMA, École normale supérieure, Université PSL, CNRS\\
75005 Paris\\
France} 
\email{\href{mailto:samuel.lerbet@ens.psl.eu}{samuel.lerbet@ens.psl.eu}}
\urladdr{\url{https://sites.google.com/view/samuel-lerbet-fr/accueil}}

\maketitle

\begin{abstract}
We study Suslin's cancellation conjecture on real affine varieties whose real locus is empty or, more generally, of small cohomological dimension.
\end{abstract}

\section{Introduction}

Let $A$ be a ring (unless explicitly stated otherwise, all rings are commutative and all modules are finitely generated in the sequel). To achieve a complete understanding of the theory of projective modules over $A$, one can divide the question in two problems: the computation of the stable theory, namely of the $\Kr$-theory group $\mathrm{\Kr}_0(A)$, which is relatively accessible as part of a cohomology theory; and the bridging of the gap between this stable problem and the unstable one, which is generally considerably more difficult. This latter problem can essentially be reduced to two questions.
\begin{itemize}
    \item The \emph{splitting} problem: given a projective $A$-module $P$, under which conditions does $P$ split off a free rank $1$ summand, that is, does there exist an $A$-module $P'$ and an isomorphism $P\simeq P'\oplus A$?
    \item The \emph{cancellation} problem: given a projective $A$-module $P$, under which conditions does $P\oplus A^{\oplus m}\simeq Q\oplus A^{\oplus m}$ imply $P\simeq Q$ for any projective $A$-module $Q$? If this implication holds, we say that $P$ is \emph{cancellative}.
\end{itemize}
In \cite{asokSplittingVectorBundles2025}, in joint work with A. Asok, the second and third authors extensively studied the splitting problem for smooth affine algebras over the field $\Rb$ of real numbers. The conclusion was that the situation is rather delicate. In particular, for vector bundles of corank $1$, it is \emph{not} controlled by the combination of the invariants which govern splitting over algebraically closed fields, namely Chern classes (\cite{murthyZeroCyclesProjective1994} and \cite{asokCohomologicalClassificationVector2014, asokSplittingVectorBundles2014, asokP1stabilizationUnstableMotivic}), and of those given by topology over the real locus. One of the positive upshots was however that if the latter disappear in the sense that the real locus has small cohomological dimension (for instance is empty), then Chern classes are sufficient to detect the splitting of a free rank $1$ summand for vector bundles of corank $\leq 1$. This idea was pursued in \cite{lerbetCohomologicalClassificationVector2026} where it was observed that on smooth real affine surfaces and threefolds of suitably cohomologically small real locus, the classification of vector bundles by Chern classes obtained in \cite{kumarAlgebraicCyclesVector1982} and \cite{asokCohomologicalClassificationVector2014} is still valid. Therefore one might wonder whether the same phenomenon occurs for the cancellation problem.

Over algebraically closed fields, recall that Suslin solved this problem in \cite{suslinCancellationTheoremProjective1977} for modules of rank $d$, where $d$ is the dimension of the base ring. More precisely, let $A$ be an affine algebra of dimension $d$ over an algebraically closed field $k$. Suslin then showed that every rank $d$ projective $A$-module is cancellative. Under the assumption that $A$ is normal (in fact smooth if $d=3$), it was shown in \cite{faselStablyFreeModules2012} that $A^{\oplus (d-1)}$ is cancellative if $(d-1)!$ is invertible in $k$, so that stably free modules of rank $d-1$ are free in this situation. The cancellation problem was essentially completely solved in the smooth case by the second author in \cite{faselSuslinsCancellationConjecture2025} where it was shown that if $d!$ is invertible in $k$ and if $A$ is smooth over $k$, then every rank $d-1$ projective $A$-module is cancellative.

Coming back to the situation over $\Rb$, the first author gave a positive result on the cancellation problem in \cite{banerjeeZeroCyclesMennicke2025} where it was shown that if $A$ is a real affine algebra of dimension $d$ without real maximal ideals, then rank $d$ projective modules are cancellative. In \cite{dasOrbitSpacesUnimodular2018}, the authors gave a count of rank $d$ stably free $A$-modules when $A$ is additionally smooth and $X(\Rb)$ is non empty and orientable which shows that these modules are free if $X(\Rb)$ has no compact connected component. In the present paper, we extend these results in the following way:

\begin{thmintro}
Let $A$ be a smooth $\Rb$-algebra of dimension $d\geq 3$; set $X=\Spec (A)$. Suppose moreover that $\Hr^d(X(\Rb),\Zb/2)=0$ (resp. and that $\Hr^{d-1}(X(\Rb),\Zb(\Lc(\Rb)))=0$ for every line bundle $\Lc$ on $X$). Then rank $d$ (resp. rank $d-1$) projective $A$-modules are cancellative.
\end{thmintro}

The assumption that a smooth manifold $M$ of dimension $d$ satisfies $\Hr^d(M,\Zb/2)=0$ is equivalent to the assertion that $M$ does not have compact connected components. Thus our result extends the cancellability of $A^{\oplus d}$, which is a consequence of the results of \cite{dasOrbitSpacesUnimodular2018}, to all rank $d$ projective $A$-modules.

\begin{remintro}
Mohan Kumar showed in \cite{kumarStablyFreeModules1985} that for every prime number $p$, there exists a smooth complex affine algebra $A$ of dimension $p+2$ such that $A^{\oplus p}$ is not cancellative. Consequently, this result (or the second author's in \cite{faselSuslinsCancellationConjecture2025}) cannot be improved in general.
\end{remintro} 

To prove this theorem, we essentially follow the method of \cite{faselSuslinsCancellationConjecture2025}. In this paper, using the general technique of motivic obstruction theory, the second author reduces the cancellability of rank $d-1$ projective modules to the following technical result, perhaps of independent interest.

\begin{thmintro}
Let $d\geq 3$ and let $X$ be a smooth real affine variety of dimension $d$; suppose that the group $\Hr^d(X(\Rb),\Zb/2)$ vanishes. Then $\Hr_\Nis^d(X,\pi_d^{\Abb^1}(\Abb^d\smallsetminus \{0\})(\Lc))=0$ for every line bundle $\Lc$ on $X$.
\end{thmintro}

In this theorem, the subscript $\Nis$ refers to the Nisnevich topology on the category of smooth $\Rb$-schemes and $\pi_d^{\Abb^1}(\Abb^d\smallsetminus \{0\})$ is the $d$-th $\Abb^1$-homotopy sheaf of $\Abb^d\smallsetminus \{0\}$, which has a $\mathbb{G}_m$-module structure and may therefore be twisted by any $\mathbb{G}_m$-torsor. In fact, for every $n\geq 1$, there exists a classifying space $\BGL_n$ of torsors under the general linear group $\GL_n$ for the Nisnevich topology. By \cite{asokAffineRepresentabilityResults2017a}, given a smooth affine $\Rb$-scheme $X$, every rank $n$ bundle on $X$ is obtained by pullback along a map of pointed spaces $X_+\rightarrow\BGL_n$ of the universal $\GL_n$-bundle on $\BGL_n$ (here $X_+$ is $X$ with a disjoint added base point), and this map is well-determined up to $\Abb^1$-homotopy. This yields a bijection $\mathscr{V}_n(X)\simeq[X_+,\BGL_n]_{\Abb^1}$, where $[\text{--},\text{--}]_{\Abb^1}$ denotes the set of maps in the pointed $\Abb^1$-homotopy category. Moreover, there is an inclusion $\GL_n\rightarrow\GL_{n+1}$ of algebraic groups carrying $M$ to $\operatorname{diag}(M,1)$; it determines a morphism $p:\BGL_n\rightarrow\BGL_{n+1}$ of classifying spaces and, modulo the previous bijection, the map $s_n$ is identified with the map $p_*:[X_+,\BGL_n]_{\Abb^1}\rightarrow[X_+,\BGL_{n+1}]_{\Abb^1}$. Thus, proving that $s_n$ is injective amounts to proving that every map $X_+\rightarrow\BGL_{n+1}$ lifts uniquely along $p$ up to $\Abb^1$-homotopy. Obstruction theory provides us with a systematic method to approach this problem by producing obstructions to lifting, and a count of these lifts, both governed by cohomology groups of the form $\Hr_\Nis^i(X,\pi_j^{\Abb^1}(F)(\xi))$ where $F$ is the homotopy fibre of $p$ and $\xi$ is a torsor under $\pi_1^{\Abb^1}(\BGL_{n+1})$. Since $F\simeq\Abb^{n+1}\smallsetminus \{0\}$ in the motivic homotopy category and $\pi_1^{\Abb^1}(\BGL_{n+1})\simeq\mathbb{G}_m$ via the determinant map, we are led precisely to studying cohomology groups such as the one appearing in the above theorem.

We also consider singular varieties in this article, for which the now established motivic techniques mentioned above do not apply. However, in the stably free case it is possible to adapt the ideas developed in \cite{faselStablyFreeModules2012} to obtain the following theorem.

\begin{thmintro}
Let $X = \Spec (A)$ be a normal affine $\mathbb{R}$-variety of dimension $d$. Assume that either $X(\mathbb{R}) = \emptyset$, or the intersection of all real maximal ideals of $A$ has height at least $2$. If $d = 3$, suppose in addition that $X$ is smooth. Then every stably free $A$-module of rank $d-1$ is free.
\end{thmintro}

Our techniques do not allow us to prove that cancellation holds for \emph{any} projective module of rank $d-1$.

\subsection*{Contents}

The organisation of the paper is as follows. In Section \ref{section:preliminaries}, we introduce the tools used in the proof of the above cancellation theorem, as well as the hypotheses $(*)$ and $(**_\Lc)$ under which cancellation is shown to hold. 
Section \ref{section:cohomological_vanishing_statements} is technical: we collect useful cohomological vanishing statements and divisibility results for cohomology groups that are essential to the arguments of \cite{faselSuslinsCancellationConjecture2025}. We then prove the main results on cancellation of vector bundles of corank $\leq 1$ in Section \ref{section:cancellation}. In the final section, we compare Chow--Witt groups and Euler class groups of smooth real affine schemes under the hypotheses $(*)$ and $(**_\Lc)$.

\subsection*{Acknowledgements} The third author was supported by ANR project CYCLADES, grant number ANR-23-CE40-0011, during his work on this project. A substantial part of this project was developed during the first author's two visits to the Institut Fourier, Université Grenoble Alpes, where many crucial discussions took place. The first visit was jointly supported by the Institut Fourier, Université Grenoble Alpes, and IMSc, while the second was supported by the Abel Visiting Scholar Program. This work was partially supported by a grant from the Niels Hendrik Abel Board. The first author sincerely acknowledge the Institut Fourier, Université Grenoble Alpes, IMSc, the Niels Hendrik Abel Board, and the IMU Commission for Developing Countries (CDC) for their generous support.

\section{Preliminaries}\label{section:preliminaries}

\subsection{The relevant hypotheses}

We start with introducing for later use three hypotheses on smooth algebraic varieties over $\Rb$. Let $X$ be such a variety and denote its dimension by $d$. We endow the real locus $X(\Rb)$ of $X$ with the Euclidean topology. The hypotheses $(*)$, $(**_\Lc)$ and $(**')$ then read as follows.
\begin{itemize}
    \item[$(*)$] The group $\Hr^d(X(\Rb),\Zb/2)$ vanishes.
    \item[$(**_\Lc)$] The groups $\Hr^d(X(\Rb),\Zb/2)$ and $\Hr^{d-1}(X(\Rb),\Zb(\Lc(\Rb)))$ vanish.
    \item[$(**')$] The groups $\Hr^d(X(\Rb),\Zb/2)$ and $\Hr^{d-1}(X(\Rb),\Zb/2)$ vanish.
\end{itemize}
In $(**_\Lc)$, the additional datum $\Lc$ is a line bundle on $X$ and the sheaf $\Zb(\Lc(\Rb))$ is the twist of the constant sheaf $\Zb$ by the $\Zb/2$-torsor underlying the real topological line bundle $\Lc(\Rb)$ on $X(\Rb)$ induced by $\Lc$ (see, \emph{e.g.}, \cite[Example 2.3]{lerbetImageHigherSignature2026} for more details). Note that the satisfaction of $(**_\Lc)$ only depends on the class of $\Lc$ in $\Pic(X)/2$. If $\Lc$ is a square, then it is omitted in the notation in the above hypotheses so we write $(**)$ for the assumption that the groups $\Hr^d(X(\Rb),\Zb/2)$ and $\Hr^{d-1}(X(\Rb),\Zb)$ are trivial.

\begin{rem}\label{rem:relevance_cohomological_assumptions}
These hypotheses appear naturally in the context of this article. Indeed $\Hr^d(X(\Rb),\Zb/2)$ is the $\Zb/2$-vector space generated by the compact connected components of $X(\Rb)$: in particular, if $(*)$ holds, then the group $\Hr^d(X(\Rb),\Zb(L))$ also vanishes for every line bundle $L$ on $X(\Rb)$ since it is Poincaré dual to a $0$-th compactly supported homology group of $X(\Rb)$ (see \cite[Corollary 1.2.2]{faselVasersteinSymbolReal2018}, for example). The relevance of the emptiness of the set of compact connected components of $X(\Rb)$ in cancellation questions is moreover made clear by \cite[Theorems 4.11 and 4.15]{dasOrbitSpacesUnimodular2018}. If $X$ is further affine, then for every line bundle $\Lc$ on $X$ with associated real topological line bundle $L$, then by \cite[Corollary 8.11]{jacobsonRealCohomologyPowers2017} and \cite[Theorem 3.5]{lerbetImageHigherSignature2026}, for every $n\geq d$, there is an isomorphism $\Hr^d(X,\Ibf^n(\Lc))\cong\Hr^d(X(\Rb),\Zb(L))$ where $\Ibf^n$ is the $n$-th power of the fundamental ideal of the Witt ring of symmetric bilinear forms and $\Ibf^n(\Lc)$ is its twist by $\Lc$. This sheaf controls the difference between Milnor--Witt and Milnor $\Kr$-theory so its appearance in motivic obstruction theory is natural. Similarly, if $n\geq d$, then by \cite[Corollary 8.11]{jacobsonRealCohomologyPowers2017} (for $n>d$) and \cite{asokSplittingVectorBundles2025} (for $n=d$), the group $\Hr^{d-1}(X,\Ibf^n(\Lc))$ is a quotient of $\Hr^d(X(\Rb),\Zb/2)\oplus\Hr^{d-1}(X(\Rb),\Zb(L))$. It follows that if $(**_\Lc)$ holds, then $\Hr^{d-1}(X,\Ibf^n(\Lc))=0$ for every $n\geq d$.
\end{rem}

In essence, the above hypotheses express that $X(\Rb)$ is ``cohomologically small''. As was observed in \cite{asokSplittingVectorBundles2025} and in \cite{lerbetCohomologicalClassificationVector2026}, for varieties $X$ with properties of this type, it is reasonable to expect that the situation becomes very similar to the known theorems over an algebraically closed base field, \emph{e.g.}, that cancellation holds for vector bundles of corank $\leq 1$.

\begin{rem}\label{rem:double_star_weaker}
Let $\Lc$ be any line bundle on $X$. Then $(**_\Lc)$ implies $(**')$. Indeed, assume that $(**_\Lc)$ holds and set $L=\Lc(\Rb)$. Since $\Hr^d(X(\Rb),\Zb/2)=0$ is the $\Zb/2$-vector space generated by the compact connected components of $X(\Rb)$, there is no such component; as noted in Remark \ref{rem:relevance_cohomological_assumptions}, this implies that $\Hr^d(X(\Rb),\Zb(L))=0$. Then the cohomology long exact sequence associated with the epimorphism $\Zb(L)\to\Zb/2$ of sheaves on $X(\Rb)$ induces an exact sequence \[\Hr^{d-1}(X(\Rb),\Zb(L))\to\Hr^{d-1}(X(\Rb),\Zb/2)\to\Hr^{d}(X(\Rb),\Zb(L))=0\] so the middle group vanishes if $\Hr^{d-1}(X(\Rb),\Zb(L))=0$, which follows from the hypothesis $(**_\Lc)$.
\end{rem}

\begin{ex}\label{exe:different_satisfaction_double_star}
The following example shows that the satisfaction of $(**_\Lc)$ for some line bundle $\Lc$ on a smooth real affine variety $X$ does not imply its satisfaction for another line bundle. Let $C\subseteq\mathbb{P}_\Rb^3$ be a smooth intersection of quadrics $Q$ and $Q'$ in $\mathbb{P}_\Rb^3$ with the following properties: the quadric $Q$ is isomorphic to $\{x_0^2+x_1^2+x_2^2=x_3^2\}$ (in particular, its real locus is diffeomorphic to the $2$-dimensional sphere $\Sr^2$) and $C(\Rb)$ is connected; set $Y=\mathbb{P}_\Rb^3\setminus C$. Since $C$ is the intersection of hypersurfaces of even degree, the orientation sheaf $\omega_{C/\mathbb{P}_\Rb^3}$ is trivial. Therefore if $\Lc\in\Pic(\mathbb{P}_\Rb^3)/2$, then setting $L=\Lc(\Rb)$, there is a localization exact sequence 
\[\begin{tikzcd}
	{\Hr^0(C(\Rb),\Zb(L_{|C}))} & {\Hr^2(\mathbb{P}^3(\Rb),\Zb(L))} & {\Hr^2(Y(\Rb),\Zb(L_{|Y}))} \\
	{\Hr^1(C(\Rb),\Zb(L_{|C}))} & {\Hr^3(\mathbb{P}^3(\Rb),\Zb(L))} & {\Hr^3(Y(\Rb),\Zb(L_{|Y}))}
	\arrow[from=1-1, to=1-2]
	\arrow[from=1-2, to=1-3]
	\arrow[from=1-3, to=2-1]
	\arrow[from=2-1, to=2-2]
	\arrow[from=2-2, to=2-3]
\end{tikzcd}\]
which also induces an isomorphism \[\Hr^0(\mathbb{P}^3(\Rb),\Zb(L))\to\Hr^0(Y(\Rb),\Zb(L_{|Y}))\] (as $C$ has codimension $2$ in $\mathbb{P}_\Rb^3$). In particular, the space $Y(\Rb)$ is connected: since $\mathbb{P}^3(\Rb)=Y(\Rb)\sqcup C(\Rb)$ where $Y(\Rb)$ and $C(\Rb)$ are both nonempty, the space $\mathbb{P}^3(\Rb)$ is connected and $C(\Rb)$ is compact, hence closed, this implies that $Y(\Rb)$ is not compact, hence does not have a compact connected component. In particular, the last group $\Hr^3(Y(\Rb),\Zb(L_{|Y}))$ in the above exact sequence is trivial. The morphism $\Hr^1(C(\Rb),\Zb(L_{|C}))\to\Hr^3(\mathbb{P}^3(\Rb),\Zb(L))$ is then an epimorphism of groups both isomorphic to $\Zb$ (if $\Lc$ is a square) or to $\Zb/2$ (if $\Lc=\Osc(1)$ modulo $2\Pic(\mathbb{P}_\Rb^3)$) and is therefore an isomorphism. In particular, we obtain an exact sequence \[\Hr^0(C(\Rb),\Zb(L_{|C}))\to\Hr^2(\mathbb{P}^3(\Rb),\Zb(L))\to\Hr^2(Y(\Rb),\Zb(L_{|Y}))\to 0\] of abelian groups.

If $\Lc=\Osc(1)$, then $\Hr^2(\mathbb{P}^3(\Rb),\Zb(L))=0$ (see \cite[Lemma 1]{cadekCohomologyBON1999}) and thus $\Hr^2(Y(\Rb),\Zb(L_{|Y}))=0$ according to this exact sequence. Now assume that $\Lc$ is a square. We claim that $\Hr^2(Y(\Rb),\Zb)=\Zb/2$. To prove this, it is sufficient (and actually equivalent) to prove that the morphism $\Hr^0(C(\Rb),\Zb)\to\Hr^2(\mathbb{P}^3(\Rb),\Zb)$ is the zero map. The map $\Hr^2(\mathbb{P}^3(\Rb),\Zb)\to\Hr^2(\mathbb{P}^3(\Rb),\Zb/2)$ is an isomorphism: indeed, as $\Hr^2(\mathbb{P}^3(\Rb),\Zb)=\Zb/2$ and $\Hr^3(\mathbb{P}^3(\Rb),\Zb)\simeq\Zb$ is torsion free, this easily follows from the Bockstein exact sequence associated with the epimorphism $\Zb\to\Zb/2$. Therefore to prove that the morphism $\Hr^0(C(\Rb),\Zb)\to\Hr^2(\mathbb{P}^3(\Rb),\Zb)$ is the zero map, it suffices to prove that the corresponding morphism $\Hr^0(C(\Rb),\Zb/2)\to\Hr^2(\mathbb{P}^3(\Rb),\Zb/2)$ with $\Zb/2$-coefficients is the zero map. This morphism factors as \[\Hr^0(C(\Rb),\Zb/2)\to\Hr^1(Q(\Rb),\Zb/2)\to\Hr^2(\mathbb{P}^3(\Rb),\Zb/2)\] by functoriality and the middle group is trivial since $Q(\Rb)\simeq\Sr^2$. This establishes the claim.

Now the variety $Y$ is not affine. However, there exists an affine open subscheme $X$ of $Y$ such that $X(\Rb)=Y(\Rb)$ (see, \emph{e.g.}, \cite[Lemma 4.1.7]{asokSplittingVectorBundles2025}). Then $\Hr^2(X(\Rb),\Zb(\Osc(1)_{|X}))=0$ but $\Hr^2(X(\Rb),\Zb)=\Zb/2$. In fact, if $d$ is any odd integer and $C$ is the smooth intersection of $d-1$ quadrics in $\mathbb{P}_\Rb^d$, one of which has real locus diffeomorphic to $\Sr^{d-1}$, and is such that $C(\Rb)$ is connected, then its complement $Y$ will satisfy $\Hr^{d-1}(Y(\Rb),\Zb(\Osc(1)_{|Y}))=0$ and $\Hr^{d-1}(Y(\Rb),\Zb)=\Zb/2$; any affine open subscheme $X$ of $Y$ such that $X(\Rb)=Y(\Rb)$ will then have the same property.
\end{ex}

\subsection{Motivic homotopy theory}

Let $k$ be a perfect field in which $2$ is invertible (eventually, the field $k$ will be real closed). We denote by $\mathsf{P}(\Sm_k)$ the $\infty$-category of \emph{pointed} presheaves of spaces on the category $\Sm_k$ of smooth separated finite-type $k$-schemes (henceforth, all (maps of) presheaves and spaces are assumed to be pointed so we drop this epithet). As examples, each presheaf of sets on $\Sm_k$ (resp. each space) determines a presheaf of spaces on $\Sm_k$; if $X\in\Sm_k$, we may therefore consider the disjoint union $X_+$ of $X$ (which represents a presheaf $U\mapsto\Hom_{\Sm_k}(U,X)$ of sets on $\Sm_k$) and an added base point $*$, yielding an object of $\mathsf{P}(\Sm_k)$.

We denote by $\Shv_\Nis(\Sm_k)$ the full subcategory of $\mathsf{P}(\Sm_k)$ spanned by those presheaves having Nisnevich descent, namely by Nisnevich sheaves of spaces. For example, Nisnevich sheaves of sets (such as $U_+$ for every $U\in\Sm_k$, since the Nisnevich topology is subcanonical) have Nisnevich descent. The inclusion of $\Shv_\Nis(\Sm_k)$ into $\mathsf{P}(\Sm_k)$ has a left adjoint $L_\Nis$, namely Nisnevich sheafification, which commutes with finite limits and in fact exhibits $L_\Nis$ as an $\infty$-topos. If $X$ and $Y$ are presheaves of spaces, we denote by $[X,Y]_\Nis$ the set of morphisms from $L_\Nis X$ to $L_\Nis Y$ in the homotopy category of $\Shv_\Nis(\Sm_k)$. We denote by $\Omega X=*\times_X *$ the loop space of $X$, and by $\Omega^nX$ the $n$-th iteration of this construction. We also denote by $\Sigma X=*\sqcup_X *$ the suspension of $X$ and by $\Sigma^n X$ the $n$-th iteration of this construction.

A \emph{motivic space} on $\Sm_k$ is a Nisnevich sheaf of spaces $X$ which is $\Abb^1$-invariant: for every $U\in\Sm_k$, the projection map $U\times\Abb^1\rightarrow U$ induces an equivalence $X(U)\simeq X(U\times\Abb^1)$. The inclusion $\Spc_k\subseteq\mathsf{P}(\Sm_k)$ of the full subcategory spanned by motivic spaces has a left adjoint $L_{\Abb^1}$ which commutes with products. We denote by $\Hr(k)$ the homotopy category of $\Spc_k$ and, given presheaves of spaces $X$ and $Y$, we denote by $[X,Y]_{\Abb^1}$ the set of maps from $L_{\Abb^1}X$ to $L_{\Abb^1}Y$ in $\Hr(k)$.

\begin{ex}
If $X$ is a motivic space, so is $\Omega X$. In other words, taking loop spaces preserves $\Abb^1$-invariance. This follows from the fact that we have an equivalence $\Map(Y,\Omega X)\cong\Omega\Map(Y,X)$ of $\infty$-groupoids natural in $Y$.
\end{ex}

Given a Nisnevich sheaf $X$ of spaces on $\Sm_k$, it has a well-defined $n$-th homotopy sheaf $\pi_n(X)$ for every $n\geq 0$ as an object in the $\infty$-topos $\Shv_\Nis(\Sm_k)$ \cite[Definition 6.5.1.1]{lurieHigherToposTheory2009}. If $X$ is any presheaf of spaces on $\Sm_k$, we then denote by $\pi_n^{\Abb^1}(X)$ the $n$-th homotopy sheaf of the Nisnevich sheaf $L_{\Abb^1}X$, so that $\pi_n^{\Abb^1}(X)$ is the Nisnevich sheaf on $\Sm_k$ associated with the presheaf $U\mapsto[\Sigma^n U_+,X]_{\Abb^1}$.

\subsubsection*{Group actions}

Let $\Gbf$ be a Nisnevich sheaf of groups acting on $V\in\Shv_\Nis(\Sm_k)$ as in \cite[Definition 3.1]{nikolausPrincipalinftybundlesGeneral2015}. We can then form the homotopy quotient $V\sslash\Gbf$ of $V$ by the action of $\Gbf$. The \emph{classifying space} $\Br\Gbf\in\Shv_\Nis(\Sm_k)$ of $\Gbf$ is defined as the homotopy quotient $\Br\Gbf=*\sslash\Gbf$. The terminal map $V\to *$ induces a map $V\sslash\Gbf\to *\sslash\Gbf=\Br\Gbf$ and the canonical map $V\to V\sslash\Gbf$ fits in a fibre sequence \[V\to V\sslash\Gbf\to \Br\Gbf.\] Every $\Gbf$-equivariant map $f:V'\to V$ induces a map $V'\sslash\Gbf\to V\sslash\Gbf$ of homotopy quotients which is an equivalence if $f$ is an equivalence of the underlying Nisnevich sheaves of spaces; this can be seen by a diagram chase in the ladder of long exact sequences of homotopy sheaves induced by the previous fibre sequence. In particular, a $\Gbf$-equivariant map $*\to V$ induces a map $*\sslash\Gbf=\Br\Gbf\to V\sslash\Gbf$ thanks to which one may regard $V\sslash\Gbf$ as a pointed object of $\Shv_\Nis(\Sm_k)/\Br\Gbf$.

Let $\Gbf$ be a Nisnevich sheaf of groups on $\Sm_k$ and let $V$ and $P$ be objects of $\Shv_\Nis(\Sm_k)$ with an action of $\Gbf$. Then there is an induced diagonal action on $P\times V$ and we set $P\times_\Gbf V=(P\times V)\sslash\Gbf$. Suppose now that $P\to X$ is a principal $\Gbf$-bundle classified by a map $\xi:X\to\Br\Gbf$, so that the map $P\to X$ identifies $X$ with the homotopy quotient $P\sslash\Gbf$ (the induced map $X\to *\sslash\Gbf=\Br\Gbf$ being precisely $\xi$). The projection $P\times V\to P$ is $\Gbf$-equivariant by definition of the diagonal action on $P\times V$ so there is an induced map $(P\times V)\sslash\Gbf\to P\sslash\Gbf=X$ on homotopy quotients thanks to which we can regard $(P\times V)\sslash\Gbf=P\times_\Gbf V$ as a space over $X$. Then according to \cite[Proposition 4.17]{nikolausPrincipalinftybundlesGeneral2015}, there is a natural equivalence 
\begin{equation}\label{eq:equivalence_lifts}
\Map_{/X}(X,P\times_\Gbf V)\simeq\Map_{/\Br\Gbf}(X,V\sslash\Gbf)  
\end{equation}
of $\infty$-groupoids (where $\Map_{/Y}$ stands for the mapping space in the slice category $\Shv_\Nis(\Sm_k)_{/Y}$). In particular, up to homotopy, a section of the map $P\times_\Gbf V=(P\times V)\sslash\Gbf\rightarrow P\sslash\Gbf=X$ is the same as a lift of the map $X\rightarrow\Br\Gbf$ along the map $V\sslash\Gbf\rightarrow\Br\Gbf$.

\subsubsection*{Strongly and strictly $\Abb^1$-invariant sheaves}

Let $X$ be a Noetherian scheme and let $\Gbf$ be a Nisnevich sheaf of groups on the category $\Sm_X$ of smooth separated $X$-schemes. We say that $\Gbf$ is \emph{strongly $\Abb^1$-invariant} if for every $j\in\{0,1\}$ and every $U\in\Sm_X$, the map $p_U^*:\Hr_\Nis^j(U,\Gbf)\to\Hr_\Nis^j(\Abb_U^1,\Gbf)$ induced by the projection $p_U:\Abb_U^1\to U$ is an isomorphism. In case $\Gbf$ is a sheaf of \emph{abelian} groups, we say that $\Gbf$ is \emph{strictly} $\Abb^1$-invariant if the map $p_U^*:\Hr_\Nis^j(U,\Gbf)\to\Hr_\Nis^j(\Abb_U^1,\Gbf)$ is an isomorphism for every $j\geq 0$. We will essentially manipulate this notion when $X\in\Sm_k$. When $X$ is the spectrum of $k$, as we review below, this definition can be advantageously recast in terms of Eilenberg--Mac Lane spaces. We say that $X$ is of $\Abb^1$-cohomological dimension $\leq d$ if for every $j>d$ and every strictly $\Abb^1$-invariant sheaf $\Abf$ on $\Sm_X$, the group $\Hr^j(X,\Abf)$ vanishes.

Let $\Gbf$ be a Nisnevich sheaf of groups on $\Sm_k$. The fibre sequence \[\Gbf\rightarrow *\rightarrow\Br\Gbf\] then guarantees that $\Br\Gbf$ is an Eilenberg--Mac Lane of type $\Kr(\Gbf,1)$: the space $\Br\Gbf$ is connected and we have $\pi_1(\Br\Gbf)=\pi_0(\Gbf)=\Gbf$ and $\pi_j(\Br\Gbf)\simeq\pi_{j-1}(\Gbf)=0$ if $j>1$. Moreover, if $X\in\Shv_\Nis(\Sm_k)$, then the $\infty$-groupoid of $\Gbf$-torsors on $X$, namely of maps $P\rightarrow X$ together with an action of $\Gbf$ on $P$ such that $P\sslash\Gbf\simeq X$, is equivalent to the $\infty$-groupoid of maps $X\rightarrow\Br\Gbf$ (via the pullback of the torsor $*\rightarrow\Br\Gbf$). In particular, the set of isomorphism classes of Nisnevich-locally trivial $\Gbf$-torsors is in bijection with $[X,\Br\Gbf]_{\Nis}=\Hr^1(X,\Gbf)$. Since $\pi_i(\Br\Gbf)=\Hr^{1-i}(\text{--},\Gbf)$ as presheaves on $\Sm_k$, we then see that $\Gbf$ is strongly $\Abb^1$-invariant if, and only if, the space $\Br\Gbf$ is $\Abb^1$-invariant.

If moreover $\Abf$ is a sheaf of \emph{abelian} groups on $\Sm_k$, then for every $n\geq 0$, there is an \emph{Eilenberg--Mac Lane} space $\Kr(\Abf,n)\in\Shv_\Nis(\Sm_k)$ with the defining property that $\pi_n(\Kr(\Abf,n))\cong\Abf$ and $\pi_m(\Kr(\Abf,n))=*$ if $m\neq n$. This space has the very desirable property that it represents cohomology: there is an isomorphism $[U_+,\Kr(\Abf,n)]_{\Nis}\cong\Hr_\Nis^n(U,\Abf)$, natural in $U\in\Sm_k$. Therefore $\Abf$ is strictly $\Abb^1$-invariant in the prvevious sense if, and only if, the Eilenberg--Mac Lane space $\Kr(\Abf,n)$ is $\Abb^1$-invariant for every $n\geq 0$. If $X\in\Shv_\Nis(\Sm_k)$, we set $\Hr^n(X,\Abf)=[X_+,\Kr(\Abf,n)]_{\Nis}$ and $\widetilde{\Hr}^n(X,\Abf)=[X,\Kr(\Abf,n)]_{\Nis}$.

We then have the following deep theorem of Morel.

\begin{thm}[Morel, \protect{\cite{morelA1AlgebraicTopologyField2012}}]\label{theo:strictly_invariant}
Let $X$ be a presheaf of spaces. Then $\pi_1^{\Abb^1}(X)$ is strongly $\Abb^1$-invariant. Moreover, if $\Gbf$ is a strongly $\Abb^1$-invariant Nisnevich sheaf of abelian groups, then $\Gbf$ is strictly $\Abb^1$-invariant. Consequently, the sheaf $\pi_n^{\Abb^1}(X)=\pi_1^{\Abb^1}(\Omega^{n-1}X)$ is strictly $\Abb^1$-invariant for every $n\geq 2$.
\end{thm}

\begin{ex}
The sheaf $\mathbb{G}_m$ is strictly $\Abb^1$-invariant. Indeed since $\mathbb{G}_m$ is abelian, by Morel's theorem cited above, it suffices to check that it is strongly $\Abb^1$-invariant. The claim then follows from the fact that $A[t]^\times=A^\times$ if $A$ is a domain (for example, the ring of global sections of a connected smooth affine $k$-scheme), and from the $\Abb^1$-invariance of the Picard group on regular schemes.
\end{ex}

The full subcategory $\Ab^{\Abb^1}(k)$ of the category $\Ab(k)$ of Nisnevich sheaves of abelian groups spanned by strictly $\Abb^1$-invariant sheaves is abelian and the fully faithful embedding $\Ab^{\Abb^1}(k)\hookrightarrow\Ab(k)$ is exact \cite[Corollary 6.24]{morelA1AlgebraicTopologyField2012}.

\subsubsection*{The twisted theory}

Let $\Gbf$ be a sheaf of groups on $\Sm_k$, let $\Abf$ be an abelian sheaf on $\Sm_k$ endowed with an additive action of $\Gbf$, and let $n$ be a non-negative integer. Then by functoriality of Eilenberg--Mac Lane spaces, the sheaf $\Gbf$ acts on $V=\Kr(\Abf,n)$, yielding a homotopy quotient $V\sslash\Gbf$: we denote it by $\Kr^\Gbf(\Abf,n)$. Note that if $\Gbf$ and $\Abf$ are strongly $\Abb^1$-invariant, then $\Kr^\Gbf(\Abf,n)$ sits in a fibre sequence $\Kr(\Abf,n)\to\Kr^\Gbf(\Abf,n)\to\Br\Gbf$ where $\Kr(\Abf,n)$ and $\Br\Gbf$ are $\Abb^1$-invariant and $\Br\Gbf$ is connected so $\Kr^\Gbf(\Abf,n)$ is $\Abb^1$-invariant by \cite[Lemma 6.51]{morelA1AlgebraicTopologyField2012} (see \cite[Lemma 2.2.10]{asokSimplicialSuspensionSequence2017} for a more detailed presentation of Morel's argument). Let $X$ be a Nisnevich sheaf of spaces with a $\Gbf$-torsor $P\to X$ represented by a map $\xi:X\to\Br\Gbf$. Then by \cite[Lemma B.15]{morelA1AlgebraicTopologyField2012}, we have a natural identification $P\times_G\Kr(\Abf,n)\simeq\Kr_X(\Abf(\xi),n)$ in the $\infty$-topos $\Shv_\Nis(\Sm_k)/X$ of sheaves over $X$, where $\Abf(\xi)=P\times_G\Abf$ is the twist of $\Abf$ by $\xi$ in the usual sense. Thus the equivalence  \eqref{eq:equivalence_lifts} specializes to an identification \[\Map_{/X}(X,\Kr_X(\Abf(\xi),n))\simeq\Map_{/\Br\Gbf}(X,\Kr^\Gbf(\Abf,n))\] of pointed $\infty$-groupoids. At the level of $\pi_0$, this reads as an isomorphism 
\begin{equation}\label{eq:cohomology_lift_through_bg}
\widetilde{\Hr}^n(X,\Abf(\xi))\cong[X,\Kr^\Gbf(\Abf,n))]_{\Nis/\Br\Gbf}
\end{equation}
of abelian groups. The element $0$ of the left-hand side corresponds to the homotopy class of the morphism $X\xrightarrow{\xi}\Br\Gbf\to\Kr^\Gbf(\Abf,n)$ in the right-hand side under this identification.

\begin{rem}\label{rem:equivariant_loop_space}
The functor $Y\mapsto Y\sslash\Gbf$ commutes with pullbacks. Indeed, the forgetful functor from $\Gbf$-actions to $\Shv_\Nis(\Sm_k)$ commutes with pullbacks; the claim then follows from the universality of colimits in the $\infty$-topos $\Shv_\Nis(\Sm_k)$ (see also \cite[proof of Lemma 4.5]{nikolausPrincipalinftybundlesGeneral2015}). For example, given an abelian sheaf $\Abf$  with additive $\Gbf$-action and $n\geq 1$, we have an equivalence $\Kr(\Abf,n-1)\to\Omega\Kr(\Abf,n)$ which is $\Gbf$-equivariant so that we have an equivalence $\Omega\Kr(\Abf,n)\sslash\Gbf\simeq\Kr(\Abf,n-1)\sslash\Gbf$. Together with the commutation of the homotopy quotient functor with pullback, this yields equivalences \[\Br\Gbf\times_{\Kr^\Gbf(\Abf,n)}\Br\Gbf\simeq(*\times_{\Kr(\Abf,n)}*)\sslash\Gbf\simeq(\Omega\Kr(\Abf,n))\sslash\Gbf\simeq\Kr(\Abf,n-1)\sslash\Gbf\simeq\Kr^\Gbf(\Abf,n-1)\] of sheaves over $\Br\Gbf$. 
\end{rem}

\begin{rem}
Let $\Gbf$ be a Nisnevich sheaf of groups acting additively on a strictly $\Abb^1$-invariant sheaf $\Abf$ and let $\xi:X\to\Br\Gbf$ be a Nisnevich-locally trivial $\Gbf$-torsor on $X\in\Sm_k$. The sheaf $\Abf(\xi)$ on $\Sm_X$ is then strictly $\Abb^1$-invariant. To see this, one takes a Nisnevich cover $U\to X$ such that $\xi_{|U}$ is trivial, so that $\Abf(\xi)_{|U}$ is isomorphic to $\Abf_{|U}$ as a sheaf on $\Sm_U$ and is therefore strictly $\Abb^1$-invariant, and uses the Čech-to-cohomology spectral sequence (which is natural in the given scheme). In particular, if $X$ has $\Abb^1$-cohomological dimension $\leq d$, then $\Hr_\Nis^j(X,\Abf(\xi))=0$ for every $j>d$.
\end{rem}

\subsubsection*{Fibre sequences}

The following situation occurs frequently when dealing with fibre sequences. Let \[G\to Y'\to Y\to Y''\] be a sequence of pointed sets, where $G$ is a group (pointed by its neutral element) acting on the set $Y'$ on the right and the map $G\to Y'$ is $G$-equivariant, hence of the form $g\mapsto y_0g$ for some $y_0\in Y'$. We say that this sequence is exact if the following assertions hold.
\begin{itemize}
    \item The fibre of $Y\to Y''$ over the base point of $Y''$ is the image of the map $Y'\to Y$.
    \item The composite $G\to Y'\to Y$ is the trivial map and elements $\alpha$ and $\beta$ of $Y'$ have the same image in $Y$ if, and only if, there exists $g\in G$ such that $\beta=\alpha g$.
\end{itemize}
These conditions imply that the fibre of $Y\to Y''$ over the base point of $Y''$ is identified by the map $Y'\to Y$ with the set $Y'/G$ of orbits of $Y'$ under the action of $G$. Note that this exact sequence provides no information on the other fibres of the map $Y\to Y''$. We now specialise the discussion to a situation in which more can be said about more general fibres of such maps.

Let $\pi:E\to B$ be a map of motivic spaces, let $X\in\Spc_k$ and let $\xi:X\to E$ be a map inducing a map $\pi_*\xi:X\to B$. 
Then by \cite[Lemma 5.5.5.12]{lurieHigherToposTheory2009}, there is a fibre sequence \[\Omega_{\pi_*\xi}\Map(X,B)\to(\Map_{/B}(X,E),\xi)\to(\Map(X,E),\xi)\to(\Map(X,B),\pi_*\xi)\] of spaces where $\Map_{/B}(X,E)$ is the mapping space from $X$ to $E$ in the $\infty$-category $\Spc_k/B$ (the space $X$ is seen as a space over $B$ through the map $\pi_*\xi$). This space parametrises lifts of $\pi_*\xi$ along $\pi$ and may therefore be pointed by $\xi$. There is an induced sequence \[\pi_0(\Omega_{\pi_*\xi}\Map(X,B))\to\pi_0(\Map_{/B}(X,E))\to\pi_0(\Map(X,E),\xi)\xrightarrow{\pi_*}\pi_0(\Map(X,B),\pi_*\xi)\] of sets which is exact in the sense described above. Suppose now that $B$ is an abelian group object in the $\Abb^1$-homotopy category $\Hr(k)$. Then $\Map(X,B)$ is an abelian group object in the homotopy category of spaces so translation along $\pi_*\xi$ induces a homotopy equivalence \[t_\xi:\Omega_*\Map(X,B)\xrightarrow{\simeq}\Omega_{\pi_*\xi}\Map(X,B)\] whose homotopy class is well-determined. In particular, the set $\pi_0(\Map_{/B}(X,E))$ inherits an action of $\pi_0(\Omega_*(X,B))$ in this manner and the fibre of $\pi_*:[X,E]_{\Nis}\to[X,B]_{\Nis}$ over the homotopy class of $\pi\circ\xi$ is identified with the quotient $\pi_0(\Map_{/B}(X,E))/\pi_0(\Omega_*\Map(X,B))$. It follows that if the composite map \[q_\xi:\pi_0(\Omega_*\Map(X,B))\xrightarrow{(t_\xi)_*}\pi_0(\Omega_{\pi_*\xi}\Map(X,B))\to\pi_0(\Map_{/B}(X,E))\] is surjective, then the fibre $\pi_*^{-1}(\pi_*\xi)$ is trivial. 

\subsubsection*{Fibre sequences in motivic homotopy theory}

An \emph{$\Abb^1$-fibre sequence} is a sequence \[F\rightarrow E\rightarrow B\] of Nisnevich sheaves of spaces such that the composite $F\rightarrow B$ is null-homotopic and such that the above sequence becomes a fibre sequence in $\Shv_\Nis(\Sm_k)$ after applying $L_{\Abb^1}$. If we have such a sequence and $F'$ is $\Abb^1$-equivalent to $F$ (namely if $L_{\Abb^1}F$ is equivalent to $L_{\Abb^1}F'$), we will sometimes abuse notation and language by saying that we have an $\Abb^1$-fibre sequence \[F'\rightarrow E\rightarrow B.\] The following example will play an important role in the sequel:

\begin{ex}\label{exe:fibre_sequence_BGLn}
We have $\pi_0^{\Abb^1}(\GL_{n+1})=\mathbb{G}_m$ \cite[Example 6.52, 2)]{morelA1AlgebraicTopologyField2012}, which is strictly $\Abb^1$-invariant as noted previously. Thus by \cite[Proposition 8.11]{morelA1AlgebraicTopologyField2012}, the inclusion $\GL_n\rightarrow\GL_{n+1}$ carrying $M$ to $\diag(M,1)$ determines an $\Abb^1$-fibre sequence \[\GL_{n+1}/\GL_n\rightarrow\BGL_n\rightarrow\BGL_{n+1}.\] On the other hand, the projection $\GL_{n+1}\rightarrow\Abb^{n+1}\smallsetminus \{0\}$ (the latter being pointed by $(0,\ldots,0,1)$) onto the last row induces an $\Abb^1$-equivalence $\GL_{n+1}/\GL_n\simeq\Abb^{n+1}\smallsetminus \{0\}$ (the fibres are Zariski-locally isomorphic to affine space so the map is an equivalence in $\Spc_k$). Thus we have an $\Abb^1$-fibre sequence 
\begin{equation}\label{eq:fibre_sequence_classifying_spaces}
\Abb^{n+1}\smallsetminus \{0\}\rightarrow\BGL_n\rightarrow\BGL_{n+1}
\end{equation}
for every $n$.
\end{ex}

\subsubsection*{Moore--Postnikov towers}

This formalism allows one to study lifting problems using cohomological tools, and is available in $\Abb^1$-homotopy theory as follows.

\begin{thm}\label{theo:moore--postnikov}
Let $k$ be a perfect field. Let $\pi:E\rightarrow B$ be a map of (pointed) motivic spaces, where $B$ is $\Abb^1$-connected; denote by $\Gbf$ the sheaf $\pi_1^{\Abb^1}(B)$. Let $F$ be the fibre of $p$ and suppose that $F$ is $\Abb^1$-simply connected (namely $\pi_i^{\Abb^1}(F)=*$ for every $i\leq 1$). Then there exist motivic spaces $E_n$ for $n\geq 0$, where $E_0$ is $\Abb^1$-equivalent to $B$, and for every $n$ a commutative diagram
\begin{center}
\begin{tikzcd}
                                       & E_{n+1} \arrow[rd,"p_{n+1}"] \arrow[d,"q_{n+1}"] & \\
E \arrow[r,swap,"i_n"] \arrow[ru,"i_{n+1}"] & E_n \arrow[r,swap,"p_n"]                          & B
\end{tikzcd}
\end{center}
in $\Spc_k$, such that the following assertions hold true.
\begin{itemize}
	\item For every $n\geq 0$, there is a commutative triangle
	\begin{center}
	\begin{tikzcd}
	                                    & E_n \arrow[rd,"p_n"] & \\
	E \arrow[rr,swap,"\pi"] \arrow[ru,"i_n"] & & B
	\end{tikzcd}
	\end{center}
	\item For every $n\geq 0$, the fibre of the map $i_n$ is $\Abb^1$-$n$-connected.
	\item For every $n\geq 0$, the fibe of the map $p_n$ is $\Abb^1$-$n$-truncated.
	\item We have $E\simeq\lim E_n$.
\end{itemize}
For every $n\geq 0$, the map $q_{n+1}:E_{n+1}\to E_n$ fits in a fibre sequence \[\Kr(\pi_{n+1}^{\Abb^1}(F),n+1)\to E_{n+1}\xrightarrow{q_{n+1}}E_n\] and in a pullback square
\begin{equation}\label{eq:twisted_principal_fibration}
\begin{tikzcd}
E_{n+1} \arrow[d,swap,"q_{n+1}"] \arrow[r] & \Br\Gbf \arrow[d] \\
E_n \arrow[r,swap,"\kappa_{n+1}"]           & \Kr^\Gbf(\pi_{n+1}^{\Abb^1}(F),n+2)
\end{tikzcd}
\end{equation}
of motivic spaces over $\Br\Gbf$, where the bottom horizontal map is called the \emph{$\kappa_{n+1}$-invariant} of $\pi$ and the right vertical map is the inclusion of the base point of $\Kr^\Gbf(\pi_{n+1}^{\Abb^1}(F),n+2)$. Such data is called a Moore--Postnikov tower of $\pi:E\to B$; it is natural in $\pi$.
\end{thm}

\begin{rem}
Since $F$ is $\Abb^1$-$1$-connected by assumption, the sheaf $\Gbf$ is also the first $\Abb^1$-homotopy sheaf of $E$ and therefore acts on the $\Abb^1$-homotopy sheaves of $F$.
\end{rem}

\begin{rem}\label{rem:first_nontrivial_stage}
Let $m$ denote the smallest integer $n$ such that $\pi_n^{\Abb^1}(F)$ is non-trivial. Then for every $n<m-1$ the map $q_{n+1}:E_{n+1}\rightarrow E_n$ sits in a fibre sequence with fibre $\Kr(\pi_{n+1}^{\Abb^1}(F),n+1)\simeq *$ and is therefore an equivalence. Hence the map $E_n\rightarrow B$ is an equivalence for $n\leq m-1$. We call $E_{m}$ the first non-trivial stage of the Moore--Postnikov tower. It fits in a cartesian square
\begin{center}
\begin{tikzcd}
E_{m} \arrow[d] \arrow[r] & \Br\Gbf \arrow[d] \\
B \arrow[r]                 & \Kr^\Gbf(\pi_{m}^{\Abb^1}(F),m+1)
\end{tikzcd}
\end{center}
of motivic spaces.
\end{rem}

The Moore--Postnikov tower allows us to understand the lifting of maps along $\pi$ up to $\Abb^1$-homotopy, that is, the image of the map $\pi_*:[X,E]_{\Abb^1}\to[X,B]_{\Abb^1}$ for $X\in\Spc_k$, as follows. Let $f:X\to B$ be a map of spaces. The first non-trivial stage of the $\Abb^1$-Postnikov tower for $B$ determines a canonical morphism $B\to\Br\pi_1^{\Abb^1}(B)=\Br\Gbf$, so we obtain a map $\xi:X\to\Br\Gbf$ by composition with $f$. For ease of notation, let us set $\Abf=\pi_{n+1}^{\Abb^1}(F)$. Suppose that a lift $f_n:X\to E_n$ of $f$ along $p_{n}$ is constructed; note that the induced map $X\to E_n\to\Br\Gbf$ is then the original $\Gbf$-torsor $\xi$ on $X$. The problem of lifting $f_n$ to $E_{n+1}$ can be analysed using the pullback square (\ref{eq:twisted_principal_fibration}). In view of this square, such a lift exists if, and only if, the composite of $f_n$ with the $k$-invariant $\kappa_{n+1}:E_n\to\Kr^\Gbf(\Abf,n+2)$ factors through the inclusion $\Br\Gbf\to\Kr^\Gbf(\Abf,n+2)$ of the base point. The composite $u:X\to\Kr^\Gbf(\Abf,n+2)$ is a morphism over $\Br\Gbf$ so it can be viewed as a cohomology class $\omega$ in $\widetilde{\Hr}^{n+2}(X,\pi_{n+1}^{\Abb^1}(F)(\xi))$ by the identification (\ref{eq:cohomology_lift_through_bg}). Then $u$ lifts through the inclusion $\Br\Gbf\to\Kr^\Gbf(\pi_{n+1}^{\Abb^1}(F),n+2)$ if, and only if, the class $\omega$ vanishes.

Assume that the obstruction class $\omega$ vanishes. Then $\kappa_{n+1}\circ f_n$ factors through $\xi:X\to\Br\Gbf$ so $f_n$ lifts along $q_{n+1}$. The space $\Map_{/E_n}(X,E_{n+1})$ of lifts of $f_n$ along $q_{n+1}$ can then be identified with the space $\Map_{/X}(X,X\times_{E_n} E_{n+1})$ of sections of the map $X\times_{E_n}E_{n+1}\rightarrow X$. The fibre product square (\ref{eq:twisted_principal_fibration}) yields an equivalence $X\times_{E_n} E_{n+1}\simeq X\times_{\Kr^\Gbf(\Abf,n+2)}\Br\Gbf$ over $X$. We note that \[X\times_{\Kr^\Gbf(\Abf,n+2)}\Br\Gbf\simeq X\times_{\Br\Gbf}(\Br\Gbf\times_{\Kr^\Gbf(\Abf,n+2)}\Br\Gbf)\simeq X\times_{\Br\Gbf}\Kr^\Gbf(\Abf,n+1)\] according to Remark \ref{rem:equivariant_loop_space}. We then have an identification \[\Map_{/X}(X,X\times_{\Br\Gbf}\Kr^\Gbf(\Abf,n+1))\simeq\Map_{/\Br\Gbf}(X,\Kr^\Gbf(\Abf,n+1))\simeq\Map_{/X}(X,\Kr_{X}(\Abf(\xi),n+1))\] using again (\ref{eq:equivalence_lifts}). Now \cite[Lemma 5.5.5.12]{lurieHigherToposTheory2009} yields a fibre sequence \[\Map_{/E_n}(X,E_{n+1})\to\Map(X,E_{n+1})\to\Map(X,E_n)\] of mapping spaces, the homotopy fibre being taken over the vertex $f_n$ of $\Map(X,E_n)$. In view of the equivalence $\Map_{/E_n}(X,E_{n+1})\simeq\Map_{/X}(X,\Kr_X(\Abf(\xi),n+1))$ of spaces described previously, this fibre sequence induces a sequence of pointed sets \[\cdots\to\pi_0(\Omega_{f_n}\Map(X,E_n))\to\widetilde{\Hr}^{n+1}(X,\pi_{n+1}^{\Abb^1}(F)(\xi))\to[X,E_{n+1}]_{\Abb^1}\xrightarrow{(q_{n+1})_*}[X,E_n]_{\Abb^1}\] that is exact in the usual sense. This sequence identifies the set of $\Abb^1$-homotopy classes of maps $f_{n+1}:X\to E_{n+1}$ such that $q_{n+1}\circ f_{n+1}$ is homotopic to $f_n$ with the set of orbits of $\widetilde{\Hr}^{n+1}(X,\pi_{n+1}^{\Abb^1}(F)(\xi))$ under the action of the group $\pi_0(\Omega_{f_n}\Map(X,E_n))$.

We will use the following elementary consequence of the Moore--Postnikov formalism in the sequel.

\begin{lem}\label{lem:easy_moore--postnikov}
Let $X$ be a smooth $k$-scheme having $\Abb^1$-cohomological dimension $\leq d$ and let $\pi:E\to B$ be a morphism of motivic spaces with fibre $F$ such that $F$ is $\Abb^1$-$i$-connected with $i\geq 1$. Then the map $\pi_*:[X_+,E]_{\Abb^1}\to[X_+,B]_{\Abb^1}$ is surjective if $i\geq d-1$ and bijective if $i\geq d$.
\end{lem}

The hypothesis of the above lemma is satisfied if $X$ has Krull dimension $\leq d$ since the Nisnevich cohomological dimension of a smooth $k$-scheme is bounded above by its Krull dimension.

\begin{proof}
Consider the Moore--Postnikov tower $(E_n,q_n)_n$ of $\pi$ (omitting the other notations from Theorem \ref{theo:moore--postnikov}). First assume that $i\geq d-1$ and let us prove that $\pi_*$ is surjective. Since $E\simeq\lim E_n$, it suffices to prove that the map $(q_{n+1})_*:[X_+,E_{n+1}]_{\Abb^1}\to[X_+,E_n]_{\Abb^1}$ is surjective for every $n\geq 2$ (the map $q_1$ being an equivalence). The obstructions to lifting a map $X\to E_n$ along $q_{n+1}$ up to homotopy live in $\widetilde{\Hr}^{n+2}(X,\pi_{n+1}^{\Abb^1}(F)(\xi))$ where $\xi$ is the $\pi_1^{\Abb^1}(B)$-torsor determined by the map $X\to E_n\to\Br\pi_1^{\Abb^1}(E_n)\simeq\Br\pi_1^{\Abb^1}(B)$. If $n\leq i-1$, then $\pi_{n+1}^{\Abb^1}(F)=0$ so the vanishing of obstruction classes is automatic. If $n\geq d-1$, then $\widetilde{\Hr}^{n+2}(X,\pi_{n+1}^{\Abb^1}(F)(\xi))=0$ by the cohomological dimension assumption on $X$. This shows that $(q_{n+1})_*$ is surjective.

Suppose now that $i\geq d$ and let us show that $\pi_*$ is bijective. As before, it suffices to prove that $(q_{n+1})_*$ is bijective for any $n$. We already know that these maps are surjective. Given a map $f_n:X\to E_n$, the set of $\Abb^1$-homotopy classes of lifts of $f_n$ along $q_{n+1}$ is a quotient of $\widetilde{\Hr}^{n+1}(X,\pi_{n+1}^{\Abb^1}(F)(\xi))$, and the same argument as before shows that, by the assumption that $i\geq d$, the group $\widetilde{\Hr}^{n+1}(X,\pi_{n+1}^{\Abb^1}(F)(\xi))$ vanishes. Consequently, the set of lifts of $f_n$ up to homotopy is a singleton; since $f_n:X\to E_n$ is arbitrary, it implies that $(q_{n+1})_*$ is injective as required.
\end{proof}

\subsubsection*{Rost--Schmid complexes}

One of the principal features of strictly $\Abb^1$-invariant sheaves is that they admit an explicit Gersten-type resolution called the \emph{Rost--Schmid complex}. If $F$ is a presheaf of spaces on $\Sm_k$, we denote by $F_{-1}$ the \emph{contraction} of $F$ defined as the internal mapping space $F_{-1}=\underline{\Hom}(\Gm,F)$ in $\mathsf{P}(\Sm_k)$. The contraction has an obvious $\mathbb{G}_m$-action, and $\mathbb{G}_m$ acts by group sheaf automorphisms on the contractions of an abelian sheaf on $\Sm_k$ (pointed in the evident manner). If $\Abf$ is a sheaf of abelian groups on $\Sm_k$, then the $\Gm$-module $\Abf_{-1}$ can therefore be twisted by any $\Gm$-torsor. More precisely, let $\Lc$ be a line bundle on an essentially smooth $k$-scheme $X$ (that is, a cofiltered limit of smooth $k$-schemes with affine étale transition maps) and let $\Lc^0$ denote the sheaf of invertible sections of $\Lc$, which is a $\Gm$-torsor. The formula $(U\rightarrow X)\mapsto\Abf_{-1}(U)\otimes_{\Zb[\mathbb{G}_m(U)]}\Zb[\Lc^0(U)]$ determines a presheaf on the site $\Sm_k/X$ of $X$-schemes that are smooth over $k$ and we denote the associated Nisnevich sheaf by $\Abf_{-1}(\Lc)$. Since $k$ is perfect, every field extension of $k$ of finite transcendence degree determines (via its spectrum) an essentially smooth $k$-scheme so $\Abf_{-1}(K,\Lc)$ is defined if $K$ is the residue field of some object of $\Sm_k/X$ at one of its points and $\Lc$ is a $K$-vector space of dimension $1$. The Rost--Schmid complex then takes the following form: we have \[\Cr_\RS^c(X,\Abf)=\bigoplus_{x\in X^{(c)}}\Abf_{-c}(\kappa(x),\omega_{x/X})\] where $X^{(c)}$ is the set of codimension $c$ points of $X$, $\Abf_{-c}$ is the $d$-fold contraction of $\Abf$ and $\omega_{x/X}=\Lambda^c(\mathfrak{m}_x/\mathfrak{m}_x^2)$ is the fibre at $x$ of the conormal sheaf of the inclusion $x\hookrightarrow X$ (thus $\omega_{x/X}$ is canonically trivial if $x$ has codimension $0$, so that $\Abf_{-0}(\kappa(x),\omega_{x/X})=\Abf(\kappa(x))$ makes sense). The construction $\Cr_\RS^*(\text{--},\Abf)$ is obviously functorial in open immersions, and in fact also in smooth morphisms so it defines presheaves $\Cr_\RS^*(\text{--},\Abf)$ on the Zariski site $X_{\Zar}$ of open subschemes of $X$ and on the small Nisnevich site of $X$. In fact, it is not difficult to check that these are flasque sheaves, and localisation at the codimension $0$ points determines morphisms $\Abf_{|X_{\Zar}}\rightarrow\mathscr{C}_\RS^0(\text{--},\Abf)_{|X_{\Zar}}$ and $\Abf_{|X_{\Nis}}\rightarrow\mathscr{C}_\RS^0(\text{--},\Abf)_{|X_{\Nis}}$. A deep theorem of Morel \cite[Corollary 5.43]{morelA1AlgebraicTopologyField2012} then asserts that $\mathscr{C}_\RS^*(\text{--},\Abf)$ is a flasque resolution of the restriction of $\Abf$ to $X_{\Zar}$ and to $X_{\Nis}$. Therefore the Rost--Schmid complex may be used to compute Nisnevich and Zariski sheaf cohomology; in particular, these cohomologies actually coincide, so that $\Hr_{\Nis}^*(X,\Abf)\cong\Hr^*(\Cr_\RS(X,\Abf))\cong\Hr_{\Zar}^*(X,\Abf)$.

\begin{ex}
The unramified Milnor $\Kr$-theory sheaf $\Kbf_n^\Mr$ is strictly $\Abb^1$-invariant for every $n$, and its contractions are computed as $(\Kbf_n^\Mr)_{-m}\cong\Kbf_{n-m}^\Mr$. If $n\leq d$, its Rost--Schmid complex, which in fact coincides with the Rost complex of \cite{rostChowGroupsCoefficients1996}, terminates as \[\bigoplus_{x\in X^{(n-1)}}\Kr_1^\Mr(\kappa(x))\xrightarrow{\partial}\bigoplus_{x\in X^{(n)}}\Kr_0^\Mr(\kappa(x))\rightarrow 0\] (the $\mathbb{G}_m$-action on Milnor $\Kr$-theory is trivial so no twist appears). Modulo the usual identifications $\Kbf_0^\Mr(F)=\Zb$ and $\Kbf_1^\Mr(F)=F^\times$ for every field $F$, the residue $\partial$ is the usual divisor map defining Chow groups so $\CH^n(X)=\Hr_{\Nis}^n(X,\Kbf_n^\Mr)=\Hr_{\Zar}^n(X,\Kbf_n^\Mr)$.
\end{ex}

If $\Abf\simeq\Bbf_{-1}$ is itself a contraction, we have an entirely similar story after twisting by a line bundle: we have a twisted Rost--Schmid complex for any line bundle $\Lc$ on the essentially smooth $k$-scheme $X$, of the form \[\Cr_\RS^c(X,\Abf;\Lc)=\bigoplus_{x\in X^{(c)}}\Abf_{-c}(\kappa(x),\omega_{x/X}\otimes\Lc(x))\] (where $\Lc(x)$ is the fibre of $\Lc$ at $x$), and this twisted Rost--Schmid complex determines a flasque resolution of the twist $\Abf(\Lc)=\Bbf_{-1}(\Lc)$. Hence $\Hr_{\Nis}^*(X,\Abf(\Lc))\cong\Hr^*(\Cr_\RS(X,\Abf;\Lc))\cong\Hr_{\Zar}^*(X,\Abf(\Lc))$.

In the sequel, we will essentially manipulate the cohomology of strictly $\Abb^1$-invariant sheaves. For such sheaves, it does not matter whether we use the Nisnevich or Zariski topology to compute cohomology, the result being the cohomology of the (twisted) Rost--Schmid complex in both cases. We therefore drop any subscript indicating the topology for the cohomology of such sheaves. On the other hand, we will have to use the étale topology for the computation of the cohomology of some sheaves (mostly tensor powers of sheaves of roots of unity); in these cases, we will always indicate that the étale topology is used with the subscript $\et$.

\subsubsection*{Milnor--Witt $\Kr$-theory}

Recall from \cite[Chapter 3]{morelA1AlgebraicTopologyField2012} the Milnor--Witt $\Kr$-theory sheaf $\Kbf_*^\MW$ of $\Zb$-graded rings. It is defined by generators and relations over fields and extended to a sheaf on $\Sm_k$ by the technology of unramified sheaves. One of the central features of this sheaf is that there is a $\Gm$-equivariant exact sequence 
\begin{equation}\label{eq:exact_sequence_mw_ktheory}
0\to\Ibf^{*+1}\to\Kbf_*^\MW\to\Kbf_*^\Mr\to 0
\end{equation}
of $\Zb$-graded strictly $\Abb^1$-invariant sheaves. Here $\Ibf^*$ is the graded sheaf whose $n$-th component $\Ibf^n$ is the $n$-th power of the fundamental ideal $\Ibf$ of the Witt ring $\mathbf{W}$ of symmetric bilinear forms \cite[Example 3.34]{morelA1AlgebraicTopologyField2012}.

\subsubsection*{Milnor $\Kr$-theory and étale cohomology}

Let $F$ be a field of characteristic $0$ and let $\ell$ be a prime number. Using symbols in Galois cohomology, one can define a \emph{norm residue homomorphism} \[h_n:\Kr_n^\Mr(F)/\ell\rightarrow\Hr_\et^n(F,\mu_\ell^{\otimes n})\] given by $\{a\}\mapsto(a)\in\Hr_\et^1(F,\mu_\ell)=F^\times/(F^\times)^\ell$ when $n=1$ and extended to any $n$ so that the maps $h_*$ assemble into a graded ring homomorphism. By the affirmation of the Milnor and Bloch--Kato conjectures (see in particular \cite{voevodskyMotivicCohomology2coefficients2003,haesemeyerNormResidueTheorem2019}), this morphism is an isomorphism. It is compatible with residue homomorphisms and therefore defines an isomorphism $\Kbf_n^\Mr/\ell\cong\Hbf^n(\ell)$ of unramified sheaves where $\Hbf^n(\ell)$ is the sheaf on $\Sm_k$ associated with the cycle module $F\mapsto\Hr_\et^n(F,\mu_\ell^{\otimes n})$.

\subsection{Grothendieck--Witt groups}

As we work in characteristic unequal to $2$, it will be convenient to use Schlichting's construction \cite{schlichtingHermitianKtheoryDerived2017} of Hermitian $\Kr$-theory, also known as higher Grothendieck--Witt groups. For $X\in\Sm_k$, we consider the category $\Dr^b(X)$ of strictly perfect complexes, namely of locally free $\Osc_X$-modules of finite rank. We endow it with the duality $\sharp$ given by $(\Esc_\bullet)^\sharp=\underline{\Hom}_{\Dr^b(X)}(\Esc_\bullet,\Osc_X[0])$ in which $\Osc_X[0]$ denotes $\Osc_X$ regarded as a complex concentrated in degree $0$, and with the biduality isomorphism $\mathrm{can}:\Esc_\bullet\to\Esc_\bullet^{\sharp\sharp}$ given by $\mathrm{can}(x)(f)=(-1)^{|x||f|}f(x)$. Together with the collection of quasi-isomorphisms of strictly perfect complexes, these data define a dg-category with weak equivalences and duality in the sense of \cite{schlichtingHermitianKtheoryDerived2017}. To such a structure, Schlichting attaches for each $n\in\Zb$ a spectrum $\mathscr{GW}^n(X)$ called the $n$-th shifted Grothendieck--Witt spectrum of $X$. We then write $\GWr_i^n(X)=\pi_i(\mathscr{GW}^n(X))$ for the $i$-th homotopy group of this spectrum: these are Schlichting's higher Grothendieck--Witt groups of $X$. These groups are contravariantly functorial with respect to morphisms of $k$-schemes between smooth $k$-schemes. We also denote by $\GW_i^n$ the sheafification of the presheaf $U\mapsto\GWr_i^n(U)$ on $\Sm_k$. The sheaf $\GW_i^n$ is the $i$-th $\Abb^1$-homotopy sheaf of the motivic space $\underline{\mathcal{GW}}^n$ representing $n$-shifted Grothendieck--Witt groups constructed by Hornbostel in \cite{hornbostelA1representabilityHermitianKtheory2005} and is in particular strictly $\Abb^1$-invariant by Morel's theorem \ref{theo:strictly_invariant}. We have $(\GW_{i+1}^{n+1})_{-1}\cong\GW_{i}^{n}$ by \cite[Proposition 4.4]{asokCohomologicalClassificationVector2014}; in particular, there is a canonical additive action of $\mathbb{G}_m$ on the contracted sheaf $\GW_i^n$.

Grothendieck--Witt groups are a ``symmetric'' refinement of algebraic $\Kr$-theory and can be compared to it in two ways. First, there is a forgetful morphism that only remembers the weak equivalences of the category of strictly perfect complexes; this determines a morphism $f:\mathscr{GW}^n(X)\to\mathscr{K}(X)$ to the connective algebraic $\Kr$-theory spectrum of $X$. On the other hand, for every $n$, any strictly perfect complex $\Esc_\bullet$ on $X$ induces a hyperbolic form $H(\Esc_\bullet)$ for the $n$-th shifted duality on $\Dr^b(X)$ and we obtain a morphism $H:\mathscr{K}(X)\to\mathscr{GW}^{n+1}(X)$ by functoriality. Finally, the tensor product of complexes determines a cup-product structure on Grothendieck--Witt groups. Then there exists a canonical element $\eta\in\GWr_{-1}^{-1}(k)$ such that cup-product with $\eta$ fits in an exact triangle \[\mathscr{GW}^n(X)\xrightarrow{f}\mathscr{K}(X)\xrightarrow{H}\mathscr{GW}^{n+1}(X)\xrightarrow{\eta\cup}\Sr^1\wedge\mathscr{GW}^n(X)\] in the homotopy category of spectra \cite[Theorem 6.1]{schlichtingHermitianKtheoryDerived2017} that is natural in $X$. Taking homotopy groups, this yields an exact sequence \[\cdots\to\GWr_i^{n}(X)\xrightarrow{f}\Kr_i(X)\xrightarrow{H}\GWr_i^{n+1}(X)\xrightarrow{\eta}\GWr_{i-1}^n(X)\to\cdots\] of abelian groups called the \emph{Karoubi periodicity} exact sequence. This exact sequence can be sheafified for the Nisnevich topology to yield an exact sequence \[\cdots\to\GW_i^{n}\xrightarrow{f}\Kbf_i\xrightarrow{H}\GW_i^{n+1}\xrightarrow{\eta}\GW_{i-1}^n\to\cdots\] of strictly $\Abb^1$-invariant sheaves on $\Sm_k$.

\subsection{Some homotopy sheaves of motivic spheres}

With our use of obstruction theory applied to the $\Abb^1$-fibre sequence (\ref{eq:fibre_sequence_classifying_spaces}) and our comparisons of (weak) Euler class groups and Chow(--Witt) groups in mind, we review some facts on $\Abb^1$-homotopy sheaves of motivic spheres. We recall that the simplicial circle $\Sr^1=\Delta^1/\partial\Delta^1$ and the geometric circle $\mathbb{G}_m$ define a bigraded sphere $\Sr^{p,q}=\Sr^p\wedge\mathbb G_m^{\wedge q}$ for every $p,q\geq 0$. For example, we have pointed $\Abb^1$-equivalences $\Sr^{n-1,n}\simeq\Abb^n\smallsetminus \{0\}$ \cite[§3, Example 2.20]{morelA1homotopyTheorySchemes1999} and $\Sr^{n,n}\simeq Q_{2n}$ \cite[Theorem 2.2.5]{asokSmoothModelsMotivic2017} where $Q_{2n}$ is the affine scheme given by $Q_{2n}=\Spec k[x_1,\ldots,x_n,y_1,\ldots,y_n,z]/\langle x_1y_1+\cdots+x_ny_n-z(1-z)\rangle$. Thus there is a pointed $\Abb^1$-equivalence $Q_{2n}\simeq\Sr^1\wedge(\Abb^n\smallsetminus \{0\})$.

\begin{thm}\label{thm:homotopy_sheaves_of_motivic_spheres}
We have the following computations of $\Abb^1$-homotopy sheaves of motivic spheres.
\begin{enumerate}
    \item There is an exact sequence \[0\to\Tbf_4'\to\pi_2^{\Abb^1}(\Abb^2\smallsetminus \{0\})\to\GW_3^2\to 0\] of sheaves on $\Sm_k$. The sheaf $\Tbf_4'$ fits in an exact sequence \[\Ibf^5\to\Tbf_4'\to\Sbf_4'\to 0\] where $\Sbf_4'$ is a quotient of the sheaf $\Kbf_4^\Mr/12$.
    \item There is an exact sequence \[0\to\Fbf_5\rightarrow\pi_3^{\Abb^1}(\Abb^3\smallsetminus \{0\})\to\GW_4^3\to 0\] of sheaves on $\Sm_k$. The sheaf $\Fbf_5$ is a quotient of a sheaf $\Tbf_5$, where the sheaf $\Tbf_5$ sits in an exact sequence of the form \[\Ibf^6\rightarrow\Tbf_5\rightarrow\Sbf_5\rightarrow 0\] and the sheaf $\Sbf_5$ is a quotient of $\Kbf_5^\Mr/24$.
    \item Let $n\geq 4$ and suppose that $k$ has characteristic $0$. Then there is an exact sequence \[0\rightarrow\Kbf_{n+2}^\Mr/24\rightarrow\pi_n^{\Abb^1}(\Abb^n\smallsetminus \{0\})\rightarrow\GW_{n+1}^n\] of sheaves on $\Sm_k$. Moreover, letting $\Qbf$ be the cokernel of the homomorphism $\pi_n^{\Abb^1}(\Abb^n\smallsetminus \{0\})\rightarrow\GW_{n+1}^n$, the $(n-3)$-fold contraction $\Qbf_{-(n-3)}$ is the zero sheaf.
    \item Let $n\geq 3$. Then the suspension morphism $\pi_n^{\Abb^1}(\Abb^n\smallsetminus\{0\})\rightarrow\pi_{n+1}^{\Abb^1}(Q_{2n})$ is an epimorphism. If $n\geq 4$, then this morphism is an isomorphism.
\end{enumerate}
Moreover, the above exact sequences and morphisms are equivariant with respect to $\Gm$-actions.
\end{thm}

\begin{proof}
For 1., 2. and 3., one uses \cite[Theorem 3.3]{asokCohomologicalClassificationVector2014} (noting that $\operatorname{Sp}_2=\mathrm{SL}_2$ is $\Abb^1$-equivalent to $\Abb^2\smallsetminus \{0\}$ by projection onto the first row), \cite[Theorem 4.4.1]{asokSplittingVectorBundles2014} and \cite[Theorem 7.2.1]{asokP1stabilizationUnstableMotivic} respectively. For the final statement, recall that Morel's $\Abb^1$-suspension theorem \cite[Theorem 6.61]{morelA1AlgebraicTopologyField2012} implies
that if $X$ is an $\Abb^1$-$(n-2)$-connected space with $n\geq 3$, then the suspension homomorphism $\pi_i^{\Abb^1}(X)\rightarrow\pi_{i+1}^{\Abb^1}(\Sigma X)$ is an isomorphism for every $i\leq 2n-4$ and an epimorphism for $i=2n-3$. Now recall that the space $X=\Abb^n\smallsetminus \{0\}$ is $\Abb^1$-$(n-2)$ connected for every $n\geq 3$ (\cite[Theorem 6.38]{morelA1AlgebraicTopologyField2012}). This completes the proof of the fourth item. The assertion on $\Gm$-equivariance then follows from the fact that these morphisms are obtained from long exact sequences of homotopy sheaves for $\Abb^1$-fibre sequences or from suspension homomorphisms.
\end{proof}

\subsection{The Bloch--Ogus spectral sequence}\label{sec:BlochOgus}

This is a coniveau spectral sequence in étale cohomology of the form \[\Er_1^{p,q}(X)=\bigoplus_{x\in X^{(p)}}\Hr_\et^{q-p}(\kappa(x),\mu_\ell^{\otimes(q-p)})\Rightarrow\Hr_\et^{p+q}(X,\mu_\ell).\] The Bloch--Ogus theorem \cite[(4.2) Theorem]{blochGerstensConjectureHomology1974} implies that the line $q=n$ of this complex is a flasque resolution of the sheaf $\Hsc^n(\ell)$ associated with the presheaf $U\mapsto\Hr_\et^n(U,\mu_\ell^{\otimes n})$ on $X$. It follows that the Nisnevich sheafification of the presheaf $U\mapsto\Hr_\et^n(U,\mu_\ell^{\otimes n})$ coincides with $\Hbf^n(\ell)$ (we refer the reader to \emph{e.g.} \cite[Theorem 2.24]{lerbetImageHigherSignature2026} for more details on comparisons of this type). In particular, the cohomology of the latter can be studied using the Bloch--Ogus spectral sequence, since we then have $\Er_2^{p,q}(X)=\Hr^p(X,\Hbf^q(\ell))$. Here is a sample application.

\begin{lem}\label{lem:bloch_ogus_alg_closed}
Let $\ell$ be a prime number. Let $k$ be an algebraically closed field of characteristic $0$, let $Y$ be a smoothy affine $k$-scheme of dimension $d>1$ and let $n\geq d$. Then the groups $\Hr^{d}(Y,\Hbf^n(\ell))$ and $\Hr^{d-1}(Y,\Hbf^n(\ell))$ vanish.
\end{lem}

\begin{proof}
We consider the Bloch--Ogus spectral sequence \[\Er_1^{p,q}(Y)=\bigoplus_{x\in Y^{(p)}}\Hr_\et^{q-p}(\kappa(x),\mu_\ell^{\otimes(q-p)})\Rightarrow\Hr_\et^{p+q}(Y,\mu_\ell).\] For cohomological dimension reasons, the group $\Hr_\et^{q-p}(\kappa(x),\mu_\ell)$ vanishes if $q>d$ for every $x\in Y^{(p)}$ (\cite[\href{https://stacks.math.columbia.edu/tag/0F0T}{Tag 0F0T}]{stacks-project}). Thus $\Er_1^{p,q}(Y)=0$ so $\Er_2^{p,q}(Y)=\Hr^p(Y,\Hbf^q(\ell))=0$ for every $q>d$. Moreover, the only non-trivial group on the line $p+q=2d$ (resp. $p+q=2d-1$) on the $\Er_2$-page of the Bloch--Ogus spectral sequence is the group $\Hr^d(Y,\Hbf^d(\ell))$ (resp. $\Hr^{d-1}(Y,\Hbf^d(\ell))$). Then for $(p,q)$ such that $p+q\in\{2d-1,2d\}$, since $\mathrm{E}_2^{p,q}(Y)=0$ for $p>d$ (as $Y^{(p)}=\emptyset$ for $p>d$) and for $q>d$, the differentials from and into $\Er_r^{p,q}(Y)$, which have bidegree $(r,-r+1)$, automatically vanish for every $r\geq 2$. It follows that $\mathrm{E}_\infty^{d-1,d}(Y)=\Hr^{d-1}(Y,\Hbf^d(\ell))=\Hr_\et^{2d-1}(Y,\mu_\ell)$ and $\mathrm{E}_\infty^{d,d}(Y)=\Hr^d(Y,\Hbf^d(\ell))=\Hr_\et^{2d}(Y,\mu_\ell)$. But since $Y$ is affine of dimension $d>1$ over an algebraically closed field, these étale cohomology groups vanish (\cite[\href{https://stacks.math.columbia.edu/tag/0F0V}{Tag 0F0V}]{stacks-project}). This completes the proof.
\end{proof}

\subsection{Unimodular rows and Mennicke symbols}\label{subsection:Um_n}

Let $R$ be a $k$-algebra and let $v=(v_1,\ldots,v_m)\in R^m$ be a row of length $m$. We say that $v$ is \emph{unimodular} if $\sum_i v_iw_i=1$ for some row $(w_1,\ldots,w_m)\in R^m$. We denote by $\Um_{m}(R)$ the set of unimodular rows of length $m$. Observe that any such row $v$ induces a morphism $f(v):\Spec(R)\to\Abb^n\smallsetminus\{0\}$: the map $f(v)$ is entirely determined by the fact that its composite with the open immersion $\Abb^n\smallsetminus\{0\}\hookrightarrow\Abb^n$ is the morphism of affine schemes associated with the morphism of rings $k[x_1,\ldots,x_n]\to R$ carrying $x_i$ to $v_i$. Note that $\GL_m(R)$ acts on $\Um_m(R)$ by multiplication on the right, hence there is an induced action of all subgroups of $\GL_m(R)$; the subgroup of interest is often the group $\mathrm{E}_m(R)$ of elementary matrices generated by transvection matrices. Moreover, if $(v_1,\ldots,v_m)\in\Um_m(R)$, then there is an induced stably free $R$-module $P$ defined as $P=\{(y_1,\ldots,y_m),\sum_i v_iy_i=0\}\subseteq R^m$: indeed, if $\sum v_iw_i=1$, then $P\oplus R\cdot(v_1,\ldots,v_m)=R^m$.

\begin{defn}	\label{weakMennickeSymbol}
A \emph{Mennicke symbol} of length $n + 1 \ge 3$ is a pair $(G,\psi)$,
where $G$ is a group and $\psi : \Um_{n+1}(R) \to G$ is a map such that:
\begin{itemize}
	\item[(1)]  $\psi((0,\ldots,0, 1)) = 1$ and $\psi(v) = \psi(v\epsilon)$ for any $v\in \Um_{n+1}(R)$ and $\epsilon \in \mathrm{E}_{n+1}(R)$;
	\item[(2)] $\psi((b_1,\ldots, b_n, x))\psi((b_1,\ldots,b_n, y)) = \psi((b_1,\ldots,b_n, xy))$ for any two unimodular rows $(b_1,\ldots, b_n, x)$ and $(b_1,\ldots,b_n, y)$ in $\Um_{n+1}(R)$.
\end{itemize}
\end{defn}

A prototypical example of such a symbol for $n\geq 2$ is constructed as follows. First, recall from \cite[\S 3.3]{faselRemarksOrbitSets2010a} that $\Hr^n(\A^{n+1}\smallsetminus \{0\},\KM_{n+1})=\Z\cdot \theta_n$ for an explicit cycle $\theta_n$. We can then consider the map $\psi_n\colon \Um_{n+1}(R)\to \Hr^n(\Spec(R),\KM_{n+1})$ defined by
\[
v\mapsto f(v)^*(\theta_n).
\]

\begin{prop}\label{Milnor K-cohomology is a Mennicke symbol}
If $X=\Spec(R)$ is a smooth affine variety over a perfect field $k$, the pair $(\Hr^{n}(X,\KM_{n+1}),\psi_n)$ is a Mennicke symbol for all $n\ge 2$.
\end{prop}



\begin{proof}
This is established in \cite[proof of Theorem 4.1]{faselRemarksOrbitSets2010a}. 
\end{proof}


\section{Some cohomological vanishing statements}\label{section:cohomological_vanishing_statements}

\subsubsection*{General cohomological vanishing results}

\begin{lem}\label{lem:right_exactness_hd}
Let $X$ be a smooth $k$-scheme of dimension $d$. Then the functor $\Hr^d(X,\text{--}):\mathrm{Ab}(X_{\Zar})\to\mathrm{Ab}$ is right-exact.
\end{lem}

\begin{proof}
This functor is additive. By general homological algebra, it then suffices to prove that it preserves epimorphisms. The cohomology long exact sequence associated with the epimorphism $f:\Abf\rightarrow\Bbf$ of abelian sheaves on $X$ exhibits the cokernel of the map $\Hr^d(X,f):\Hr^d(X,\Abf)\to\Hr^d(X,\Bbf)$ as a subgroup of $\Hr^{d+1}(X,\ker f)$ which vanishes by \cite[Théorème 3.6.5]{grothendieckQuelquesPointsDalgebre1957} since $X$ is Noetherian of dimension $d$.
\end{proof}

\begin{cor}\label{cor:divisibility_vanishing_multiplication_by_n_non_unique}
Let $X$ be a smooth $k$-scheme of dimension $d$, let $a\geq 1$ be an integer and let $\Abf$ be an abelian sheaf on $X$. Then $\Hr^d(X,\Abf)$ is $a$-divisible if, and only if, the group $\Hr^d(X,\Abf/a)$ vanishes.
\end{cor}

\begin{proof}
Consider the endomorphism $f=\cdot a:\Abf\to\Abf$ of multiplication by $a$, whose cokernel is $\Abf/a$. Lemma \ref{lem:right_exactness_hd} yields an isomorphism $\Hr^d(X,\Abf)/a\cong\Hr^d(X,\Abf/a)$. The left-hand side is vanishes if, and only if, the group $\Hr^d(X,\Abf)$ is divisible, which completes the proof.
\end{proof}

The following technical lemma, which is an elaboration on the previous result, will allow us to streamline a few arguments in the sequel.

\begin{lem}\label{lem:unique_divisibility_vanishing}
Let $X$ be a smooth $k$-scheme of dimension $d$ and let $f:\Abf\rightarrow\Bbf$ be a morphism of abelian sheaves on $X$. Suppose that the group $\Hr^d(X,\ker f)$ vanishes. The following assertions are then equivalent.
\begin{itemize}
	\item[(1)] The map $\Hr^{d-1}(X,f):\Hr^{d-1}(X,\Abf)\rightarrow\Hr^{d-1}(X,\Bbf)$ is surjective and the map $\Hr^d(X,f):\Hr^d(X,\Abf)\rightarrow\Hr^d(X,\Bbf)$ is an isomorphism.
	\item[(2)] The groups $\Hr^{d-1}(X,\coker f)$ and $\Hr^{d}(X,\coker f)$ vanish.
\end{itemize}
\vspace{-\itemsep}
\end{lem}

\begin{proof}
Consider the exact sequences \[0\rightarrow\ker f\rightarrow\Abf\xrightarrow{f}\im f\rightarrow 0,\;\;0\rightarrow\im f\rightarrow\Bbf\rightarrow\coker f\rightarrow 0\] of abelian sheaves, that yield the following commutative diagram:
\begin{center}
\begin{tikzcd}
& \Hr^{d-1}(X,\Abf) \arrow[r] \arrow[rd,swap,"\protect{\Hr^{d-1}(X,f)}"]       & \Hr^{d-1}(X,\im f) \arrow[r,"\partial"] \arrow[d] & \Hr^{d}(X,\ker f)=0 \\
&                                                        & \Hr^{d-1}(X,\Bbf) \arrow[d]                       & \\
&                                                        & \Hr^{d-1}(X,\coker f) \arrow[d]                     & \\
\Hr^{d}(X,\ker f)=0 \arrow[r] & \Hr^d(X,\Abf) \arrow[r] \arrow[rd,swap,"\protect{\Hr^d(X,f)}"] & \Hr^d(X,\im f) \arrow[r] \arrow[d]                 & \Hr^{d+1}(X,\ker f)=0 \\
                               &                         & \Hr^d(X,\Bbf) \arrow[d]   & \\
                               &                         & \Hr^d(X,\coker f) \arrow[d] & \\
                               &                         & \Hr^{d+1}(X,\im f)=0      &
\end{tikzcd}
\end{center} 
In this diagram, the diagonal maps are the maps $\Hr^i(X,f):\Hr^i(X,\Abf)\rightarrow\Hr^i(X,\Bbf)$ ($i\in\{d-1,d\}$) induced by $f$ in cohomology. Moreover, the group $\Hr^d(X,\ker f)$ vanishes by assumption, and the groups $\Hr^{d+1}(X,\ker f)$ and $\Hr^{d+1}(X,\im f)$ vanish since $X$ is noetherian of dimension $d$ (\cite[Théorème 3.6.5]{grothendieckQuelquesPointsDalgebre1957}). The desired conclusion now follows easily from this diagram.
\end{proof}

As before, applying this lemma to the multiplication by $a$ endomorphism yields the following corollary.

\begin{cor}\label{cor:divisibility_vanishing_multiplication_by_n}
Let $X$ be a smooth $k$-scheme of dimension $d$, let $\Abf$ be an abelian sheaf on $X_{\Zar}$ and let $a\geq 1$ be an integer; denote by $\Abf[a]$ the $a$-torsion subsheaf of $\Abf$. If $\Hr^d(X,\Abf[a])=0$, then the following assertions are equivalent.
\begin{itemize}
	\item[(1)] The group $\Hr^{d-1}(X,\Abf)$ is $a$-divisible and the group $\Hr^d(X,\Abf)$ is uniquely $a$-divisible.
	\item[(2)] The groups $\Hr^{d-1}(X,\Abf/a)$ and $\Hr^{d}(X,\Abf/a)$ vanish.
\end{itemize}
\vspace{-\itemsep}\vspace{-\itemsep}
\end{cor}

\subsubsection*{Divisibility for $\Kbf$-cohomology groups}

We start with the following vanishing result.

\begin{lem}\label{lem:vanishing_no_compact_component}
Let $X$ be a smooth affine variety over $\Rb$ of dimension $d>1$ and let $n\geq d$. If $(*)$ (resp. $(**')$) is satisfied, then $\Hr^d(X,\Kbf_n^\Mr/2)=0$ (resp. and $\Hr^{d-1}(X,\Kbf_n^\Mr/2)=0$).
\end{lem}

\begin{proof}
By \cite[Theorem 2.3.2 (c)]{colliot-theleneZerocyclesCohomologyReal1996}, if $m>d$, there is an isomorphism $\Hr^i(X,\Hbf^{m}(2))\cong\Hr^i(X(\Rb),\Zb/2)$ for every $i\geq 0$. This immediately implies the lemma in case $n>d$. Moreover, \cite[Theorem 3.2 (d), Corollary 3.3 (c)]{colliot-theleneZerocyclesCohomologyReal1996} apply to the smooth affine variety $X$ and yield isomorphisms \[\Hr^d(X,\Hbf^d(2))\simeq\Hr^d(X(\Rb),\Zb/2),\;\Hr^{d-1}(X,\Hbf^d(2))\simeq\Hr^{d}(X(\Rb),\Zb/2)\oplus\Hr^{d-1}(X(\Rb),\Zb/2)\] of abelian groups. These isomorphisms guarantee that the group $\Hr^d(X,\Hbf^d(2))\simeq\Hr^d(X,\Kbf_d^\Mr/2)$ (resp. and the group $\Hr^{d-1}(X,\Hbf^d(2))\simeq\Hr^{d-1}(X,\Kbf_d^\Mr/2)$) vanishes (resp. vanish) under the hypothesis $(*)$ (resp. $(**')$).
\end{proof}

The following lemma guarantees that one of the hypotheses of Lemma \ref{lem:unique_divisibility_vanishing} holds in cases very meaningful to the present paper.

\begin{lem}\label{lem:contraction_of_torsion_milnor_k-theory}
Let $k$ be a perfect field, let $X$ be a smooth affine $k$-scheme of dimension $d$ and let $a\neq 0$ be an integer. Then $\Hr^d(X,\Kbf_d^\Mr[a])=0$. Let further $n>d$; then the group $\Hr^d(X,\Kbf_n^\Mr[a])$ vanishes in the following two situations:
\begin{itemize}
    \item[(i)] the field $k$ is algebraically closed;
    \item[(ii)] the field $k$ is $\Rb$ and the integer $a$ is odd or $a=2$ and the hypothesis $(*)$ is satisfied.
\end{itemize}
\vspace{-\topsep}
\end{lem}

\begin{proof}
The sheaf $\Kbf_d^\Mr[a]$, being the kernel of a morphism of strictly $\Abb^1$-invariant sheaves, has a Rost--Schmid complex. In degree $d$, we have $\Cr_\RS^d(X,\Kbf_d^\Mr[a])=\bigoplus_{x\in X^{(d)}}(\Kbf_d^\Mr[a])_{-d}(\kappa(x),\omega_{x/X})$. By exactness of the formation of contractions, $(\Kbf_d^\Mr[a])_{-d}=\Kbf_{d-d}^\Mr[a]=\Kbf_0^\Mr[a]=\Zb[a]=0$ as $\Zb$ is torsion free. Thus $\Cr_\RS^d(X,\Kbf_d^\Mr[a])=0$ so $\Hr^d(X,\Kbf_d^\Mr[a])=\Hr^d(\Cr_\RS(X,\Kbf_d^\Mr[a]))=0$. This completes the proof of the first assertion.

We now proceed to the second part of Lemma \ref{lem:contraction_of_torsion_milnor_k-theory}. We first recall some facts about the Milnor $\Kr$-theory of $\Rb$ and of algebraically closed fields (see for example \cite[Chapter III, Section 7, Example 7.2 (b), (c)]{weibelKbookIntroductionAlgebraic2013}). Suppose that $k$ is algebraically closed. Then $\Kr_m^\Mr(k)$ is uniquely divisible for every $m\geq 2$. Moreover, if $m\geq 2$, then $\Kr_m^\Mr(\Rb)$ is the direct sum of $\Zb/2\cdot\{-1\}^m$ (the product being taken in the graded ring $\Kr_*^\Mr(\Rb)$) and of a uniquely divisible group (generated by the symbols $\{a_1,\ldots,a_m\}$ with $a_i>0$). This shows that the $a$-torsion of $\Kr_*^\Mr(\Rb)$ coincides with its $2$-torsion if $a$ is even, that for $F\in\{\Rb,k\}$, the morphism $\cdot\{-1\}:\Kbf_{m-1}^\Mr/2\rightarrow\Kbf_m^\Mr[2]$ of sheaves on $\Sm_F$ given by multiplication by $\{-1\}$ is an epimorphism in sections over $F$, and that if $k$ is algebraically closed, then choosing a primitive $a$-th root $\xi_a$ of unit, the map $\cdot\{\xi_a\}:\Kbf_{m-1}^\Mr/a\rightarrow\Kbf_m^\Mr[a]$ of sheaves on $\Sm_k$ is surjective on sections over $k$.

Now we show that $\Hr^d(X,\Kbf_n^\Mr[a])=0$ when $n>d$ under the hypotheses of the lemma. Suppose first as in (i) that $k$ is algebraically closed. The sheaf $\Kbf_n^\Mr[a]$ is strictly $\Abb^1$-invariant so by exactness of contractions, the group $\Hr^d(X,\Kbf_n^\Mr[a])$ is a subquotient of \[\Cr_\RS^d(X,\Kbf_n^\Mr[a])=\bigoplus_{x\in X^{(d)}}\Kr_{n-d}^\Mr(\kappa(x))[a]\] where $\kappa(x)=k$ since $X$ has dimension $d$. If $n\neq d+1$, then $\Kr_{n-d}^\Mr(\kappa(x))$ is torsion-free so the group $\Cr_\RS^d(X,\Kbf_n^\Mr[a])$ vanishes and thus $\Hr^d(X,\Kbf_n^\Mr[a])=0$. Now assume that $n=d+1$. Let $\Abf$ be the cokernel of the morphism $\cdot\{\xi_a\}:\Kbf_d^\Mr/a\rightarrow\Kbf_{d+1}^\Mr[a]$; this sheaf is strictly $\Abb^1$-invariant as the cokernel of a morphism between strictly $\Abb^1$-invariant sheaves. By Lemma \ref{lem:right_exactness_hd}, we have an exact sequence \[\Hr^d(X,\Kbf_d^\Mr/a)\rightarrow\Hr^d(X,\Kbf_{d+1}^\Mr[a])\rightarrow\Hr^d(X,\Abf)\rightarrow 0\] of abelian groups. The group $\Abf(k)$ is the cokernel of the morphism $\cdot\{\xi_a\}:\Kr_0^\Mr(k)/a\to\Kr_1^\Mr(k)[a]$ so it vanishes by the argument of the previous paragraph. Therefore the group $\Hr^d(X,\Abf)$ vanishes by examination of the Rost--Schmid complex. To conclude, it suffices to prove that $\Hr^d(X,\Kbf_d^\Mr/a)=0$. To do so, it suffices to show that $\Hr^d(X,\Kbf_d^\Mr)$ is divisible and thus that it is $\ell$-divisible for every prime $\ell$ so we are reduced to showing that $\Hr^d(X,\Kbf_d^\Mr/\ell)=0$ for every prime $\ell$. The norm residue isomorphism induces an isomorphism $\Kbf_d^\Mr/\ell\simeq\Hbf^d(\ell)$ of sheaves so $\Hr^d(X,\Kbf_d^\Mr/\ell)=0$ by Lemma \ref{lem:bloch_ogus_alg_closed}. This completes the proof of item (i).

Suppose finally that $k=\Rb$. Suppose first that $a$ is odd and let us show that $\Hr^d(X,\Kbf_n^\Mr[a])=0$. We consider the étale projection map $\pi:Y=X\times_\Rb\Spec\Cb\rightarrow X$. Since $\pi$ is finite, it induces a pushforward map $\pi_*:\Hr^i(Y,\Kbf_n^\Mr/a)\rightarrow\Hr^i(X,\Kbf_n^\Mr/a)$  and the composite \[\Hr^d(X,\Kbf_n^\Mr[a])\xrightarrow{\pi^*}\Hr^d(Y,\Kbf_n^\Mr[a])\xrightarrow{\pi_*}\Hr^d(X,\Kbf_n^\Mr[a])\] is multiplication by the degree $2$ of $\pi$. Since $\Hr^i(Y,\Kbf_n^\Mr[a])=0$, it follows that multiplication by $2$ is the zero endomorphism, namely $\Hr^i(X,\Kbf_d^\Mr[a])$ is $2$-torsion. It is also $a$-torsion: since $a$ is odd, it must therefore be the zero group, as required. Suppose now that the hypothesis $(*)$ is satisfied and that $a=2$. Then as in the previous paragraph, the group $\Hr^d(X,\Kbf_n^\Mr[2])$ is a quotient of $\Hr^d(X,\Kbf_{n-1}^\Mr/2)$ and is therefore the zero group (Lemma \ref{lem:vanishing_no_compact_component}). This completes the proof.
\end{proof}

\begin{prop}\label{prop:kdivisibility_odd}
Let $X$ be a smooth affine $\Rb$-scheme of dimension $d>1$, let $n\geq d$ and let $a$ be an odd integer. Then $\Hr^{i}(X,\Kbf_n^\Mr)$ is $a$-divisible for $i\in\{d-1,d\}$.
\end{prop}

\begin{proof}
It suffices to see that $\Hr^i(X,\Kbf_n^\Mr)$ is $\ell$-divisible for every odd prime number $\ell$. By Lemma \ref{lem:contraction_of_torsion_milnor_k-theory}, the group $\Hr^d(X,\Kbf_n^\Mr[\ell])$ vanishes. By Corollary \ref{cor:divisibility_vanishing_multiplication_by_n}, it then suffices to show that the groups $\Hr^{d-1}(X,\Kbf_n^\Mr/\ell)$ and $\Hr^d(X,\Kbf_n^\Mr/\ell)$ vanish. Setting $Y=X\times_\Rb\Spec\Cb$, the groups $\Hr^{d-1}(Y,\Kbf_n^\Mr/\ell)$ and $\Hr^d(Y,\Kbf_n^\Mr/\ell)$ vanish by Lemma \ref{lem:bloch_ogus_alg_closed} (in view of the norm residue isomorphism $\Kbf_n^\Mr/\ell\cong\Hbf^n(\ell)$). The same transfer argument as in the proof of Lemma \ref{lem:contraction_of_torsion_milnor_k-theory} then shows that the groups $\Hr^{d-1}(X,\Kbf_n^\Mr/\ell)$ and $\Hr^d(X,\Kbf_n^\Mr/\ell)$ are $2$-torsion and $\ell$-torsion so they vanish as required. 
\end{proof}

\begin{prop}\label{prop:divisibility_Hd-1Kd_real}
Let $X$ be a smooth affine $\Rb$-scheme of dimension $d$ and let $n\geq d$. Suppose that the hypothesis $(*)$ is satisfied. Then $\Hr^d(X,\Kbf_n^\Mr)$ is divisible. Moreover, under the hypothesis $(**')$, the group $\Hr^{d-1}(X,\Kbf_n^\Mr)$ is also divisible.
\end{prop}

\begin{proof}
First assume that the hypothesis $(*)$ is satisfied and let us show that $\Hr^{d}(X,\Kbf_n^\Mr)$ is divisible. By Proposition \ref{prop:kdivisibility_odd}, it suffices to prove that $\Hr^d(X,\Kbf_n^\Mr)$ is $2$-divisible, and thus, by Corollary \ref{cor:divisibility_vanishing_multiplication_by_n_non_unique}, that the group $\Hr^d(X,\Kbf_n^\Mr/2)$ vanishes. This is Lemma \ref{lem:vanishing_no_compact_component}. 

We now assume that the hypothesis $(**')$ is satisfied. We have to show that the group $\Hr^{d-1}(X,\Kbf_n^\Mr)$ is $a$-divisible for every $a\neq 0$. Again by Proposition \ref{prop:kdivisibility_odd}, we see that it suffices to prove that this group is $2$-divisible. Therefore by (ii) in Lemma \ref{lem:contraction_of_torsion_milnor_k-theory} and Lemma \ref{lem:unique_divisibility_vanishing}, we are reduced to proving that $\Hr^{d-1}(X,\Kbf_n^\Mr/2)$ and $\Hr^{d}(X,\Kbf_n^\Mr/2)$ vanish. This follows from Lemma \ref{lem:vanishing_no_compact_component}.
\end{proof}

\begin{rem}
Recall the exact sequence \[0\rightarrow\Ibf^{*+1}\rightarrow\Kbf_*^\MW\rightarrow\Kbf_*^\Mr\rightarrow 0\] of abelian sheaves on $\Sm_k$ (\ref{eq:exact_sequence_mw_ktheory}). It induces an exact sequence \[\Hr^d(X,\Ibf^{d+2})\rightarrow\Hr^d(X,\Kbf_{d+1}^\MW)\rightarrow\Hr^d(X,\Kbf_{d+1}^\Mr)\rightarrow 0\] of abelian groups if $X\in\Sm_k$ has dimension $d$. Suppose now that $k=\Rb$ and that $X$ is affine and the group $\Hr^d(X(\Rb),\Zb/2)$ vanishes. Then by Jacobson's theorem \cite{jacobsonRealCohomologyPowers2017}, we have $\Hr^d(X,\Ibf^{d+2})\cong\Hr^d(X(\Rb),\Zb)$ and the latter vanishes since $X(\Rb)$ has no compact connected component. Thus the map $\Hr^d(X,\Kbf_{d+1}^\MW)\rightarrow\Hr^d(X,\Kbf_{d+1}^\Mr)$ is an isomorphism. We conclude from Proposition \ref{prop:divisibility_Hd-1Kd_real} that $\Hr^d(X,\Kbf_{d+1}^\MW)$ is divisible. Now by \cite[Theorem 4.9]{faselRemarksOrbitSets2010a} (at least if $d\geq 3$), the group $\Hr^d(X,\Kbf_{d+1}^\MW)$ is isomorphic to the group $\mathrm{WMS}_{d+1}(X)$ housing the universal weak Mennicke symbol of $X$ (we refer to \emph{loc. cit} for the definition of this group). Thus the group $\mathrm{WMS}_{d+1}(X)$ group is also divisible. In fact, given a finite type affine $\Rb$-scheme $X$, the divisibility of $\mathrm{WMS}_{d+1}(X)$ was known in the following situations:
\begin{itemize}
    \item the real locus of $X$ is empty \cite[Corollary 4.3]{banerjeeZeroCyclesMennicke2025};
    \item the variety $X$ is smooth and its real locus $X(\Rb)$ is nonempty, orientable and has no compact connected component \cite[Corollary 4.5, Theorem 4.6]{dasOrbitSpacesUnimodular2018}
\end{itemize}
The above argument gives a unified argument for these results (restricting \cite[Corollary 4.3]{banerjeeZeroCyclesMennicke2025} to smooth schemes).
\end{rem}

\subsubsection*{The top cohomology of the relevant homotopy sheaf} 

\begin{prop}\label{prop:vanishing_hd_GW_d+1_d}
Let $d\geq 2$, let $X$ be a smooth affine $\Rb$-scheme of dimension $d$ satisfying the hypothesis $(*)$ and let $\Lc$ be a line bundle on $X$. Then $\Hr^d(X,\GW_{d+1}^d(\Lc))$ vanishes.
\end{prop}

\begin{proof}
Recall the Karoubi periodicity exact sequence: 
\begin{equation}\label{eq:karoubi_periodicity_exact_sequence}
\GW_{d+1}^{d-1}\xrightarrow{f}\Kbf_{d+1}\xrightarrow{H}\GW_{d+1}^d\xrightarrow{\eta}\GW_{d}^{d-1}\xrightarrow{f}\Kbf_d    
\end{equation}
of sheaves on $\Sm_k$, inducing an exact sequence of sheaves on $X$. This sequence can be twisted by $\Lc$ and remains exact. Setting $\mathbf{A}(\Lc):=\coker(\GW_{d+1}^{d-1}(\Lc)\xrightarrow{f}\Kbf_{d+1})$ and $\mathbf{B}(\Lc):=\ker(\GW_{d}^{d-1}(\Lc)\xrightarrow{f}\Kbf_{d})$, we obtain an exact sequence
\[
0\to \mathbf{A}(\Lc)\to \GW_{d+1}^d\to \mathbf{B}(\Lc)\to 0.
\]
It suffices thus to show that $\Hr^d(X,\mathbf{A}(\Lc))=\Hr^d(X,\mathbf{B}(\Lc))=0$ to conclude. For the former, consider the following commutative diagram with exact rows
\[
\xymatrix{\GW_{d+1}^{d-1}(\Lc)\ar[r]^-f & \Kbf_{d+1}\ar[r] & \mathbf{A}(\Lc)\ar[r] & 0 \\
 \Kbf_{d+1}\ar[r]_{f\circ H}\ar[u]^-H &  \Kbf_{d+1}\ar[r]\ar@{=}[u] & \mathbf{A}'(\Lc)\ar[r]\ar@{-->}[u] & 0}
\]
in which the induced morphism $\mathbf{A}'(\Lc)\to \mathbf{A}(\Lc)$ is onto. We are reduced to show that $\Hr^d(X,\mathbf{A}'(\Lc))=0$. To this end, we first observe that the action of $\mathbb{G}_m$ on $ \mathbf{A}'$ is trivial, being the cokernel of a morphism of sheaves on which $\mathbb{G}_m$ acts trivially. Now, we observe that $(\Kbf_{d+1})_{-m}$ coincides with the unramified Milnor $K$-theory sheaf if $m\geq d-1$ and we may use \cite[Lemma 4.2]{asokCohomologicalClassificationVector2014} to deduce that the composite $(f\circ H)_{-m}=f_{-m}\circ H_{-m}$ is actually the multiplication by $2$ morphism. Consequently 
\[
\Hr^d(X,\mathbf{A}'(\Lc))=\Hr^d(X,\Kbf_{d+1}/2)=\Hr^d(X,\Kbf_{d+1}^\Mr/2)=0
\]
where the second-to-last equality follows from the agreement of Quillen and Milnor $\Kr$-theory for fields in degree $\leq 2$, and the last equality from Proposition \ref{prop:divisibility_Hd-1Kd_real}.

We finally prove that $\Hr^d(X,\mathbf{B}(\Lc))=0$. By \cite[Lemma 2.4]{faselStablyFreeModules2012}, the previous exact sequence (contracted $d-1$ times) induces a split exact sequence 
\[
0\rightarrow\Kbf_1/2\xrightarrow{H}\GW_{1}^0\xrightarrow{\eta}\GW_0^{-1}=\GW^3\rightarrow 0.
\] 
Contracting once again yields an isomorphism $\Kbf_0/2\xrightarrow{H}\GW_{0}^3$, and we deduce that the hyperbolic map $H\colon \Kbf_d\to \GW_{d}^{d-1}$ induces isomorphisms $\Kbf_1/2\to \mathbf{B}(\Lc)_{1-d}$ and $\Kbf_0/2\to \mathbf{B}(\Lc)_{-d}$ (compatible with contractions). Thus
\[
\Hr^d(X,\mathbf{B}(\Lc))\simeq \Hr^d(X,\Kbf_d/2)=0
\]
by Proposition \ref{prop:divisibility_Hd-1Kd_real}.
\end{proof}

\begin{prop}\label{prop:vanishing_cohomology_homotopy_sheaf}
Let $d\geq 2$ and let $X$ be a smooth real affine variety of dimension $d$; assume that $\Hr^d(X(\Rb),\Zb/2)=0$. Then $\Hr^d(X,\pi_d^{\Abb^1}(\Abb^d\smallsetminus \{0\})(\Lc))=0$ for every line bundle $\Lc$ on $X$.  
\end{prop}

\begin{proof}
Suppose first that $d\geq 4$. Recall from Theorem \ref{thm:homotopy_sheaves_of_motivic_spheres} the $\mathbb{G}_m$-equivariant exact sequence \[0\to\Kbf_{d+2}^\Mr/24\to\pi_d^{\Abb^1}(\Abb^d\smallsetminus \{0\})\to\GW_{d+1}^d\] of sheaves. Moreover, denoting by $\Qbf$ the cokernel of the morphism $\pi_d^{\Abb^1}(\Abb^d\smallsetminus \{0\})\to\GW_{d+1}^d$, the contraction $\Qbf_{(d-3)}$ is the zero sheaf. Now letting $\Abf$ be the image of this morphism, the sheaves $\Abf$ and $\Qbf$ are strictly $\Abb^1$-invariant so their cohomology can be computed using Rost--Schmid complexes. The complex $\Cr_\RS(X,\Qbf(\Lc))$ then vanishes in degree $\geq d-3$ by exactness of twists so in particular, one has $\Hr^j(X,\Qbf(\Lc))=0$ if $j=d-1,d$. The cohomology long exact sequence associated with the inclusion $\Abf(\Lc)\subseteq\GW_{d+1}^d(\Lc)$ then shows that it induces an isomorphism $\Hr^d(X,\Abf(\Lc))\to\Hr^d(X,\GW_{d+1}^d(\Lc))$. By Lemma \ref{lem:right_exactness_hd}, we then have an exact sequence \[\Hr^d(X,\Kbf_{d+2}^{\mathrm{M}}/24)\rightarrow\Hr^d(X,\pi_d^{\Abb^1}(\Abb^d\smallsetminus \{0\})(\Lc))\rightarrow\Hr^d(X,\Abf(\Lc))\cong\Hr^d(X,\GW_{d+1}^d(\Lc))\rightarrow 0\] of abelian groups. It now suffices to prove that the groups $\Hr^d(X,\Kbf_{d+2}^{\mathrm{M}}/24)$ and $\Hr^d(X,\GW_{d+1}^d(\Lc))$ vanish. In the second case, this was checked in Proposition \ref{prop:vanishing_hd_GW_d+1_d}. Moreover, the group $\Hr^d(X,\Kbf_{d+2}^\Mr/24)$ is the zero group by Proposition \ref{prop:divisibility_Hd-1Kd_real} (whose assumptions are satisfied by hypothesis), and Corollary \ref{cor:divisibility_vanishing_multiplication_by_n_non_unique}.

We now treat the cases $d\leq 3$. Suppose first that $d=3$. Recall from Theorem \ref{thm:homotopy_sheaves_of_motivic_spheres} the exact sequence of describing $\pi_3^{\Abb^1}(\Abb^3\smallsetminus \{0\})$. It is $\mathbb{G}_m$-equivariant hence induces an exact sequence
\begin{equation}\label{eq:description_pi_2}
0\rightarrow\Fbf_5(\Lc)\rightarrow\pi_3^{\Abb^1}(\Abb^3\smallsetminus \{0\})(\Lc)\rightarrow\GW_4^3(\Lc)\rightarrow 0 
\end{equation}
of sheaves on $X$. In this sequence, the  sheaf $\Fbf_5(\Lc)$ is a quotient of a sheaf $\Tbf_5(\Lc)$ fitting in an exact sequence
\begin{equation}\label{eq:exact_sequence_tbf4}
\Ibf^{6}(\Lc)\rightarrow\Tbf_{5}(\Lc)\rightarrow\Sbf_5\rightarrow 0
\end{equation}
and $\Sbf_5$ is a quotient of the sheaf $\Kbf_5^\Mr/24$ (actually, $\Sbf_5$ is isomorphic to $\Kbf_5^\Mr/24$ by \cite{asokMotivicspheres20}, but we don't need this fact here). By Proposition \ref{prop:divisibility_Hd-1Kd_real}, the group $\Hr^3(X,\Kbf_5^\Mr)$ is divisible hence Corollary \ref{cor:divisibility_vanishing_multiplication_by_n} yields $\Hr^3(X,\Kbf_5^\Mr/24)=0$. By Lemma \ref{lem:right_exactness_hd}, the group $\Hr^3(X,\Sbf_5)$ is a quotient of $\Hr^3(X,\Kbf_5^\Mr/24)$ and therefore vanishes. By \cite[Corollary 8.11]{jacobsonRealCohomologyPowers2017}, the group $\Hr^3(X,\Ibf^5(\Lc))$ is isomorphic to the group $\Hr^3(X(\Rb),\Zb(L))$ where $L=\Lc(\Rb)$. One has $\Hr^3(X(\Rb),\Zb(L))=0$ because of the hypothesis $(*)$ thus $\Hr^3(X,\Ibf^5(\Lc))=0$. Then the exact sequence (\ref{eq:exact_sequence_tbf4}) and Lemma \ref{lem:right_exactness_hd} imply that there is an exact sequence \[0=\Hr^3(X,\Ibf^5(\Lc))\rightarrow\Hr^3(X,\Tbf_5(\Lc))\rightarrow\Hr^3(X,\Sbf_5)=0\] hence $\Hr^3(X,\Tbf_5(\Lc))=0$. Now again the group $\Hr^3(X,\Fbf_5(\Lc))$ is a quotient of $\Hr^3(X,\Tbf_5(\Lc))$ according toLemma \ref{lem:right_exactness_hd}: consequently, we have $\Hr^3(X,\Fbf_5(\Lc))$. Finally, the exact sequence (\ref{eq:description_pi_2}) provides an exact sequence \[\Hr^3(X,\Fbf_5(\Lc))\rightarrow\Hr^3(X,\pi_3^{\Abb^1}(\Abb^3\smallsetminus \{0\})(\Lc))\rightarrow\Hr^3(X,\GW_4^3(\Lc))\rightarrow 0\] where the groups $\Hr^3(X,\Fbf_5(\Lc))$ and $\Hr^3(X,\GW_4^3(\Lc))$ vanish by the previous arguments and by Proposition \ref{prop:vanishing_hd_GW_d+1_d} respectively. Thus $\Hr^3(X,\pi_3^{\Abb^1}(\Abb^3\smallsetminus \{0\})(\Lc))$, which completes the proof in the case where $d=3$. Finally, if $d=2$, then the argument is the exact same, using the description of $\pi_2^{\Abb^1}(\Abb^2\smallsetminus \{0\})$ from Theorem \ref{thm:homotopy_sheaves_of_motivic_spheres}, and is omitted (see for instance \cite[Proposition 3.11]{lerbetCohomologicalClassificationVector2026}).
\end{proof}

Let us mention a useful consequence of this proposition. Let $d\geq 3$. Recall that $\Abb^d\smallsetminus\{0\}$ is $\Abb^1$-$(d-2)$-connected by \cite[Theorem 6.38]{morelA1AlgebraicTopologyField2012}, so the first nontrivial stage of the Postnikov tower of $\Abb^d\smallsetminus\{0\}$ (that is, the Moore--Postnikov tower of the morphism $\Abb^d\smallsetminus\{0\}\to *$) is 
\[
(\Abb^d\smallsetminus\{0\})^{(d-1)}=\Kr(\pi_{d-1}^{\Abb^1}(\Abb^d\smallsetminus\{0\}),d-1)
\] 
where $\pi_{d-1}^{\Abb^1}(\Abb^d\smallsetminus\{0\})=\Kbf_d^\MW$ (\cite[Theorem 6.40]{morelA1AlgebraicTopologyField2012}). This yields a $\kappa$-invariant $\kappa:\Abb^d\smallsetminus\{0\}\to\Kr(\Kbf_d^\MW,d-1)$, hence a map \[\kappa_*=\zeta_{d-1}(X):[X_+,\Abb^d\smallsetminus\{0\}]_{\Abb^1}\to\Hr^{d-1}(X,\Kbf_d^\MW)\] for every smooth scheme $X$, functorial in $X$ and surjective by Lemma \ref{lem:easy_moore--postnikov} if in addition $X$ has dimension $\leq d$. In fact:

\begin{lem}\label{lem:wms_bijective}
Let $X$ be a smooth $\Rb$-scheme of dimension $d$ satisfying $(*)$. Then $\zeta_{d-1}(X)$ is bijective.
\end{lem}

\begin{proof}
The morphism $\kappa$ factors as \[\Abb^d\smallsetminus\{0\}\to(\Abb^d\smallsetminus\{0\})^{(d)}\xrightarrow{q}\Kr(\pi_{d-1}^{\Abb^1}(\Abb^d\smallsetminus\{0\}),d-1)\] where the first morphism induces a bijection $[X_+,\Abb^d\smallsetminus\{0\}]_{\Abb^1}\xrightarrow{\sim}[X_+,(\Abb^d\smallsetminus\{0\})^{(d)}]_{\Abb^1}$ by Lemma \ref{lem:easy_moore--postnikov}. Therefore it suffices to prove that the map $q_*:[X_+,(\Abb^d\smallsetminus\{0\})^{(d)}]_{\Abb^1}\to\Hr^{d-1}(X,\Kbf_d^\MW)$ induced by $q$ is injective (as we already know that it is surjective). As explained in the discussion following Theorem \ref{theo:moore--postnikov}, the set of $\Abb^1$-homotopy classes of lifts of a given map $\alpha:X_+\to\Kr(\Kbf_d^\MW,d-1)$ along $q$ up to $\Abb^1$-homotopy is a quotient of the set \[\widetilde{\Hr}^d(X_+,\pi_d^{\Abb^1}(\Abb^d\smallsetminus\{0\}))=\Hr^d(X,\pi_d^{\Abb^1}(\Abb^d\smallsetminus\{0\}))\] (note that $\Abb^d\smallsetminus\{0\}$ is $\Abb^1$-simply connected so no action of the $\Abb^1$-fundamental group on higher homotopy sheaves needs to be considered). This last set is a singleton by Proposition \ref{prop:vanishing_cohomology_homotopy_sheaf} and thus its quotient $q_*^{-1}(\alpha)$ is a singleton, as required.
\end{proof}

\begin{cor}\label{cor:mennicke_bijective}
Let $X=\Spec(A)$ be a smooth affine $\Rb$-scheme of dimension $d\geq 3$ satisfying $(**)$. Then the Mennicke symbol $\psi_{d-1}:\Um_d(A)/\Er_d(A)\to\Hr^{d-1}(X,\Kbf_d^\Mr)$ is bijective.
\end{cor}

\begin{proof}
First note that (\ref{eq:exact_sequence_mw_ktheory}) induces an exact sequence \[\Hr^{d-1}(X,\Ibf^{d+1})\to\Hr^{d-1}(X,\Kbf_d^\MW)\to\Hr^{d-1}(X,\Kbf_d^\Mr)\to\Hr^{d}(X,\Ibf^{d+1})\] where $\Hr^{i}(X,\Ibf^{d+1})\cong\Hr^i(X(\Rb),\Zb)$ by \cite[Corollary 8.11]{jacobsonRealCohomologyPowers2017}. Thus $\Hr^{d}(X,\Ibf^{d+1})=0$ by the hypothesis $(*)$ (see Remark \ref{rem:relevance_cohomological_assumptions}) and $\Hr^{d-1}(X,\Ibf^{d+1})=\Hr^{d-1}(X(\Rb),\Zb)=0$ by the hypothesis $(**)$. We conclude that the morphism $\rho_{d-1}(X):\Hr^{d-1}(X,\Kbf_d^\MW)\to\Hr^{d-1}(X,\Kbf_d^\Mr)$ is an isomorphism.

Now the assignment $f:\Um_d(A)\to\Hom_\Rb(X,\Abb^d\smallsetminus\{0\})$ taking $v$ to $f(v)$ induces a bijection \[f:\Um_d(A)/\Er_d(A)\xrightarrow{\sim}[X_+,\Abb^d\smallsetminus\{0\}]_{\Abb^1}\] by \cite[Remark 8.10]{morelA1AlgebraicTopologyField2012} which is clearly functorial in $X$. To conclude, it suffices to prove that the diagram
\[\begin{tikzcd}
	{\Um_d(A)/\Er_d(A)} & {[X_+,\Abb^d\smallsetminus\{0\}]_{\Abb^1}} & {\Hr^{d-1}(X,\Kbf_d^\MW)} \\
	{\Um_d(A)/\Er_d(A)} && {\Hr^{d-1}(X,\Kbf_d^\Mr)}
	\arrow["{f}"', "\sim", from=1-1, to=1-2]
	\arrow[equal, from=1-1, to=2-1]
	\arrow["{\zeta_{d-1}}"', "\sim", from=1-2, to=1-3]
	\arrow["\wr", "\rho_{d-1}"', from=1-3, to=2-3]
	\arrow["{\psi_{d-1}}"', from=2-1, to=2-3]
\end{tikzcd}\]
commutes up to a sign (here $\zeta_{d-1}$ is bijective by Lemma \ref{lem:wms_bijective}). We let $\psi_{d-1}^\prime$ denote the clockwise composition in this diagram. Now consider the quadric $Q_{2d-1}=\{(x,y),\sum_ix_iy_i=1\}\subseteq\Abb^{d}\times\Abb^{d}$. Any unimodular row of length $d$ with coefficients in $A$ is pulled back from the universal row $x=(x_1,\ldots,x_d)\in\Um_d(\Rb[Q_{2d-1}])$ by a morphism $X\to Q_{2d-1}$. Therefore by functoriality of the preceding constructions, it suffices to prove that $\psi_{d-1}^\prime(x)=\varepsilon\psi_{d-1}(x)$ for some $\varepsilon\in\{\pm 1\}$ to complete the proof. By \cite[Theorem 1.3.4]{asokEulerClassGroups2021}, there exists a (so-called fundamental) class $\alpha_{\Abb^d\smallsetminus\{0\}}\in\Hr^{d-1}(\Abb^d\smallsetminus\{0\},\Kbf_d^\MW)$ such that $\zeta_{d-1}(g)=g^*\alpha_{\Abb^d\smallsetminus\{0\}}$ for any $g:Q_{2d-1}\to\Abb^d\smallsetminus\{0\}$. The image $\theta_{d-1}^\prime$ of $\alpha_{\Abb^d\smallsetminus\{0\}}$ in $\Hr^{d-1}(\Abb^d\smallsetminus\{0\},\Kbf_d^\Mr)\simeq\Zb$ under $\rho_{d-1}$ is then a generator, hence $\theta_{d-1}^\prime=\varepsilon\theta_{d-1}$ for some $\varepsilon\in\{\pm 1\}$. Now if $p:Q_{2d-1}\to\Abb^d\smallsetminus\{0\}$ is the projection onto the first factor, then $p=f(x)$ in the notation of §\ref{subsection:Um_n}. Thus we have \[\psi_{d-1}^\prime(x)=\rho_{d-1}(Q_{2d-1})\circ\zeta_{d-1}(Q_{2d-1})(p)=\rho_{d-1}(Q_{2d-1})(p^*\alpha_{\Abb^d\smallsetminus\{0\}})=p^*\theta_{d-1}^\prime=\varepsilon p^*\theta_{d-1}.\] But $p^*\theta_{d-1}=\psi_{d-1}(x)$ by definition, hence $\psi_{d-1}^\prime(x)=\varepsilon\psi_{d-1}(x)$ as required.
\end{proof}

\section{Cancellation in corank $\leq 1$ over smooth real affine varieties}\label{section:cancellation}

\subsection{Representability}

For $n\geq 1$, the (Nisnevich simplicial) classifying space $\BGL_n\in\Shv_\Nis(\Sm_k)$ is defined. Given $X\in\Sm_k$, we also denote by $\Vsc_n(X)$ the collection of isomorphism classes of rank $n$ vector bundles on $X$; we denote by $\{\Esc\}$ the class in $\Vsc_n(X)$ of the rank $n$ vector bundle $\Esc$, and we endow $\Vsc_n(X)$ with the base point $\{\Osc_X^{n}\}$.

\begin{thm}[Morel \protect{\cite{morelA1AlgebraicTopologyField2012}}, Schlichting \protect{\cite{schlichtingEulerClassGroups2017}}, Asok--Hoyois--Wendt \protect{\cite{asokAffineRepresentabilityResults2017a}}]\label{theo:representability}
Let $k$ be a field, let $n\geq 1$ and let $X$ be a smooth affine $k$-scheme. Then there is a bijection $[X_+,\BGL_n]_{\Abb^1}\simeq\mathscr{V}_n(X)$ of pointed sets.
\end{thm}

Consider the morphism $\GL_n\rightarrow\GL_{n+1}$ of algebraic groups taking $M$ to $\operatorname{diag}(M,1)$. This morphism induces a map $p:\BGL_n\rightarrow\BGL_{n+1}$ of classifying spaces and thus for every $X\in\Sm_k$ a map \[p_*:[X_+,\BGL_n]_{\Abb^1}\rightarrow[X_+,\BGL_{n+1}]_{\Abb^1}.\] Tracing through the definitions, it is not too difficult to see that $p_*$ corresponds to the map $s_n:\Vsc_n(X)\rightarrow\Vsc_{n+1}(X)$ taking $\{\Esc\}$ to $\{\Esc\oplus\Osc_X\}$ under the bijections provided by Theorem \ref{theo:representability}. We also denote by $\GL$ the colimit of the groups $\GL_n$ under the inclusions $\GL_n\rightarrow\GL_{n+1}$. The classifying space $\BGL$ actually represents reduced algebraic $\Kr$-theory: if $X$ is \emph{any} smooth $k$-scheme (not necessarily affine), we have an isomorphism $[X_+,\BGL]_{\Abb^1}\simeq\widetilde{\Kr}_0(X)$ \cite[§4, Theorem 3.13]{morelA1homotopyTheorySchemes1999}, where $\widetilde{\Kr}_0(X)$ is the kernel of the rank homomorphism. The map $\BGL_n\rightarrow\BGL$ induced by the structure homomorphism $\GL_n\rightarrow\GL$ induces for every $X\in\Sm_k$ a map $[X_+,\BGL_n]_{\Abb^1}\rightarrow[X_+,\BGL]_{\Abb^1}$ which, modulo the bijections $[X_+,\BGL_n]_{\Abb^1}\simeq[X_+,\BGL]_{\Abb^1}$, is the stabilisation map $s_{n,\infty}:\Vsc_n(X)\rightarrow\widetilde{\Kr_0}(X)$ taking $\{\Esc\}$ to the class $[\Esc]-[\Osc_X^n]$ in reduced $\Kr$-theory. By definition of reduced $\Kr$-theory, a rank $n$ vector bundle $\Esc$ is cancellative if, and only if, the fibre of $s_{n,\infty}$ over $[\Esc]-[\Osc_X^n]$ is the singleton reduced to the isomorphism class $\{\Esc\}$. This assertion can be studied by homotopy-theoretic methods.

\subsection{Motivic obstruction theory for vector bundles}

For $n\geq  1$, we consider the $\Abb^1$-fibre sequence \[F_n\to\BGL_n\to\BGL.\] It fits in a commutative diagram
\begin{center}
\begin{tikzcd}
F_{n-1} \arrow[r] \arrow[d] & \Br\GL_{n-1} \arrow[r] \arrow[d] & \Br\GL \arrow[d,equal] \\
F_n \arrow[r] & \Br\GL_n \arrow[r] & \Br\GL 
\end{tikzcd}
\end{center}
that determines an $\Abb^1$-fibre sequence \[\operatorname{fib}(F_{n-1}\to F_n)\to\operatorname{fib}(\Br\GL_{n-1}\to\Br\GL_n)\to *\] By Example \ref{exe:fibre_sequence_BGLn}, the $\Abb^1$-fibre of the map $\Br\GL_{n-1}\to\Br\GL_n$ is $\Abb^n\smallsetminus\{0\}$ so we have an $\Abb^1$-fibre sequence of the form 
\begin{equation}\label{eq:fibre_sequence_of_fibres}
\Abb^n\smallsetminus \{0\}\to F_{n-1}\to F_n.
\end{equation}
By \cite[Theorem 3.2]{asokAlgebraicVectorBundles2014}, the fibre $F_n$ is $\Abb^1$-$(n-1)$-connected. Then for $n\geq 2$, the sheaf $\pi_{n-1}^{\Abb^1}(\Abb^n\smallsetminus \{0\})$ is naturally isomorphic to $\Kbf_n^\MW$ by \cite[Theorem 6.40]{morelA1AlgebraicTopologyField2012} so the long exact sequence of $\Abb^1$-homotopy sheaves induced by the fibre sequence (\ref{eq:fibre_sequence_of_fibres}) is of the form 
\begin{equation}\label{eq:exact_sequence_homotopy_sheaves_fibres}
\pi_n^{\Abb^1}(\Abb^n\smallsetminus \{0\})\to\pi_n^{\Abb^1}(F_{n-1})\to\pi_n^{\Abb^1}(F_n)\to\pi_{n-1}^{\Abb^1}(\Abb^n\smallsetminus \{0\})=\Kbf_n^\MW\to\pi_{n-1}^{\Abb^1}(F_{n-1})\rightarrow\pi_{n-1}^{\Abb^1}(F_n)=*.
\end{equation}
The map $\pi_{n-1}^{\Abb^1}(\Abb^n\smallsetminus \{0\})=\Kbf_n^\MW\to\pi_{n-1}^{\Abb^1}(F_{n-1})$ of homotopy sheaves is determined in \cite[p. 15]{duEnumeratingNonStableVector2022}: it is an isomorphism if $n$ is odd and it is the reduction mod $\eta$ map $\Kbf_n^\MW\to\Kbf_n^\Mr$ if $n$ is even.

We collect some consequences of this analysis that are relevant for the sequel.

\begin{lem}\label{lem:cohomology_homotopy_sheaves_milnor}
Let $X$ be a smooth affine $\Rb$-scheme of dimension $d$, let $\Lc$ be a line bundle on $X$ and let $n\geq d$. Suppose that $\Hr^m(X(\Rb),\Zb(\Lc(\Rb)))=0$ for every $m\geq n-1$. Then the projection homomorphism $\pi_{n-1}^{\Abb^1}(F_{n-1})\to\Kbf_{n}^\Mr$ induces an isomorphism $\Hr^{n-1}(X,\pi_{n-1}^{\Abb^1}(F_{n-1})(\Lc))\cong\Hr^{n-1}(X,\Kbf_n^\Mr)$ of abelian groups.
\end{lem}

\begin{proof}
If $n$ is even, then this is clear as the map $\pi_{n-1}^{\Abb^1}(F_{n-1})\to\Kbf_n^\Mr$ is an isomorphism so we may assume $n$ odd, in which case $\pi_{n-1}^{\Abb^1}(F_{n-1})=\Kbf_n^\MW$. The exact sequence \[0\to\Ibf^{n+1}(\Lc)\to\Kbf_n^\MW(\Lc)\to\Kbf_n^\Mr\to 0\] of sheaves on $X$ shows that it suffices to prove the following vanishing statements: the cohomology groups $\Hr^{n-1}(X,\Ibf^{n+1}(\Lc))$ and $\Hr^{n}(X,\Ibf^{n+1}(\Lc))$ vanish. Since $n\geq d$, this follows from Jacobson's theorem \cite[Corollary 8.11]{jacobsonRealCohomologyPowers2017} identifying $\Hr^j(X,\Ibf^{n+1}(\Lc))$ and $\Hr^j(X(\Rb),\Zb(L))$ (where $L=\Lc(\Rb)$) and our assumption on $X(\Rb)$. 
\end{proof}

\begin{lem}\label{lem:iso_homotopy_sheaves_of_fibres}
Let $X$ be a smooth $\Rb$-scheme of dimension $d$ and let $\Lc$ be a line bundle on $X$. Suppose that $(**_\Lc)$ holds. Then the map $\pi_d^{\Abb^1}(F_{d-1})\to\pi_d^{\Abb^1}(F_d)$ induces an isomorphism $\Hr^d(X,\pi_d^{\Abb^1}(F_{d-1})(\Lc))\cong\Hr^d(X,\pi_d^{\Abb^1}(F_{d})(\Lc))$ of abelian groups.
\end{lem}

\begin{proof}
Suppose first that $d$ is even. As noted after (\ref{eq:exact_sequence_homotopy_sheaves_fibres}), in this case, there is an exact sequence \[\pi_d^{\Abb^1}(\Abb^d\smallsetminus \{0\})\to\pi_d^{\Abb^1}(F_{d-1})\to\pi_d^{\Abb^1}(F_d)\to\Kbf_d^\MW\to\Kbf_d^\Mr\to 0\] of strictly $\Abb^1$-invariant sheaves which is $\Gm$-equivariant since it is induced by an $\Abb^1$-fibre sequence. It determines an exact sequence \[\pi_d^{\Abb^1}(\Abb^d\smallsetminus \{0\})(\Lc)\to\pi_d^{\Abb^1}(F_{d-1})(\Lc)\to\pi_d^{\Abb^1}(F_d)(\Lc)\to\Ibf^{d+1}(\Lc)\to 0\] of sheaves on $X$. Let $\Qbf$ be the cokernel of the homomorphism $\pi_d^{\Abb^1}(\Abb^d\smallsetminus \{0\})\to\pi_d^{\Abb^1}(F_{d-1})$ of sheaves. By Lemma \ref{lem:right_exactness_hd}, there is an induced exact sequence \[\Hr^d(X,\pi_d^{\Abb^1}(\Abb^d\smallsetminus \{0\})(\Lc))\to\Hr^d(\pi_d^{\Abb^1}(F_{d-1})(\Lc))\to\Hr^d(X,\Qbf(\Lc))\to 0\] where $\Hr^d(X,\pi_d^{\Abb^1}(\Abb^d\smallsetminus \{0\})(\Lc))=0$ by Proposition \ref{prop:vanishing_cohomology_homotopy_sheaf}. There is also an exact sequence \[\Hr^{d-1}(X(\Rb),\Ibf^{d+1}(\Lc))\to\Hr^d(X,\Qbf(\Lc))\to\Hr^d(X,\pi_d^{\Abb^1}(F_d)(\Lc))\to\Hr^d(X,\Ibf^{d+1}(\Lc))\to 0.\] By Jacobson's theorem, the group $\Hr^j(X,\Ibf^{d+1}(\Lc))$ is isomorphic to $\Hr^j(X(\Rb),\Zb(L))$ for $j=d-1,d$ and therefore vanishes by assumption. Consequently, the maps \[\Hr^d(\pi_d^{\Abb^1}(F_{d-1})(\Lc))\to\Hr^d(X,\Qbf(\Lc)),\;\Hr^d(X,\Qbf(\Lc))\to\Hr^d(X,\pi_d^{\Abb^1}(F_d)(\Lc))\] are isomorphisms hence so is the composite $\Hr^d(X,\pi_d^{\Abb^1}(F_{d-1})(\Lc))\to\Hr^d(X,\pi_d^{\Abb^1}(F_d)(\Lc))$ as required.

Suppose now that $d$ is odd. Then (\ref{eq:exact_sequence_homotopy_sheaves_fibres}) yields an exact sequence \[\pi_{d}^{\Abb^1}(\Abb^d\smallsetminus \{0\})\to\pi_d^{\Abb^1}(F_{d-1})\to\pi_d^{\Abb^1}(F_d)\to 0\] of abelian sheaves. It induces an exact sequence \[\Hr^d(X,\pi_d^{\Abb^1}(\Abb^d\smallsetminus \{0\})(\Lc))\to\Hr^d(X,\pi_d^{\Abb^1}(F_{d-1})(\Lc))\to\Hr^d(X,\pi_d^{\Abb^1}(F_d)(\Lc))\to 0\] again by Lemma \ref{lem:right_exactness_hd}, where $\Hr^d(X,\pi_d^{\Abb^1}(\Abb^d\smallsetminus \{0\})(\Lc))=0$ by Proposition \ref{prop:vanishing_cohomology_homotopy_sheaf}. This proves that the map $\Hr^d(X,\pi_d^{\Abb^1}(F_{d-1})(\Lc))\to\Hr^d(X,\pi_d^{\Abb^1}(F_d)(\Lc))$ is an isomorphism as required.
\end{proof}

\subsection{Corank $0$}

We start with the case of cancellation in corank $0$, which is easier. Thus we fix an integer $d\geq 2$, which we think of as the dimension of the varieties under consideration, and we study the morphism $\pi_d:\BGL_d\rightarrow\BGL$. Let $X$ be a smooth affine $\Rb$-scheme. Recall that cancellation for rank $d$ vector bundles on $X$ is equivalent to the statement that the map \[(\pi_d)_*:[X_+,\BGL_d]_{\Abb^1}\to[X_+,\BGL]_{\Abb^1}\] is injective. We study this assertion by obstruction-theoretic methods. The fibre $F_d$ of $\pi_d$ being $\Abb^1$-$(d-1)$-connected, the first non-trivial stage of the Moore--Postnikov tower of $\pi_d$ will be denoted by $\widetilde{E}_{d}$ and sits in a pullback square
\begin{center}
\begin{tikzcd}
\widetilde{E}_{d} \arrow[d,swap,"\widetilde{p}_d"] \arrow[r] & \Br\mathbb{G}_m \arrow[d] \\
\BGL \arrow[r]          & \Kr^{\mathbb{G}_m}(\pi_{d}^{\Abb^1}(F),d+1)
\end{tikzcd}
\end{center}
as $\pi_1^{\Abb^1}(\BGL)=\mathbb{G}_m$ via the determinant map, inducing a sequence
\begin{equation}\label{eq:corank_0}
\pi_1(\Map(X_+,\BGL),(\widetilde{p}_d)_*(\xi))\to\Hr^d(X,\pi_d^{\Abb^1}(F_d)(\Lc))\to[X_+,\widetilde{E}_d]_{\Abb^1}\to[X_+,\BGL]_{\Abb^1}
\end{equation}
of pointed sets which is exact for every map $\xi:X_+\to\widetilde{E}$. We first make the following observation.

\begin{lem}\label{lem:bijectivity_corank_0}
The map $i_d:\BGL_d\to\widetilde{E}_d$ determined from the Moore--Postnikov tower of $\pi_d$ induces a bijection $(i_d)_*:[X_+,\BGL_d]_{\Abb^1}\xrightarrow{\cong}[X_+,\widetilde{E}_d]_{\Abb^1}$.
\end{lem}

\begin{proof}
The fibre of $i_d$ is $\Abb^1$-$d$-connected. Since $X$ has dimension $d$, the claim then follows from Lemma \ref{lem:easy_moore--postnikov}.
\end{proof}

Thus the injectivity of $(\pi_d)_*$ is equivalent to the injectivity of $(\widetilde{p}_d)_*:[X_+,\widetilde{E}_d]_{\Abb^1}\to[X,\BGL]_{\Abb^1}$. 

\begin{lem}\label{lem:injectivity_corank_0}
Suppose that the hypothesis $(*)$ is satisfied. Then the map $(\widetilde{p}_d)_*$ is injective.
\end{lem}

\begin{proof}
Let $\xi:X_+\to\widetilde{E}_d$ be a map; there is an induced map $(\widetilde{p}_d)_*\xi:X_+\to\BGL$ by composition with $\widetilde{p}_d$ and this map determines a reduced $\Kr$-theory class of $X$ with a well-defined determinant $\Lc$. The line bundle $\Lc$ is the $\mathbb{G}_m$-torsor corresponding to the composite map $X_+\to\BGL\to\Br\pi_1^{\Abb^1}(\BGL)$, recalling that $\pi_1^{\Abb^1}(\BGL)=\mathbb{G}_m$ via the determinant morphism $\GL\to\mathbb{G}_m$. The space $\BGL$ is a group object in the $\Abb^1$-homotopy category $\Hr(k)$ by \cite[Corollary 5.6]{duEnumeratingNonStableVector2022} so we can consider translation along $(\widetilde{p}_d)_*(\xi)$, which is well-defined at the level of homotopy groups and yields an isomorphism \[\pi_0(\Omega_*\Map(X_+,\BGL))\xrightarrow{\cong}\pi_0(\Omega_{(\widetilde{p}_d)_*\xi}\Map(X_+,\BGL))\] of abelian groups. We also recall that we have an identification $\Hr^d(X,\pi_d^{\Abb^1}(F_d)(\Lc))\cong\Hr^d(X,\Kbf_{d+1}^\Mr)$ (Lemma \ref{lem:cohomology_homotopy_sheaves_milnor}). Thus we may rewrite (\ref{eq:corank_0}) as follows: \[\pi_0(\Omega_*\Map(X,\BGL))=[X,\GL]_{\Abb^1}\xrightarrow{\Delta_d(\xi)}\Hr^d(X,\Kbf_{d+1}^\Mr)\to[X_+,\widetilde{E}_d]_{\Abb^1}\xrightarrow{(\widetilde{p}_d)_*}[X_+,\BGL]_{\Abb^1}.\] To show that $(\widetilde{p}_d)_*^{-1}((\widetilde{p}_d)_*\xi)$ is a singleton, it now suffices to prove that $\Delta_d(\xi)$ is surjective. This map was computed in \cite[Proposition 6.3]{duEnumeratingNonStableVector2022} using Chern classes of matrices: we have \[\Delta_d(\xi)(M)=c_{d+1}(M)+\sum_{j=1}^dc_j(M)c_{d+1-j}(\xi)\] for $M\in[X_+,\GL]_{\Abb^1}=\Kr_1(X)$. We now consider the Suslin matrix morphism \[S_{n}:Q_{2n-1}\to\GL\] of spaces defined in \cite[§3.5]{asokAlgebraicVectorBundles2014} using the explicit construction of \cite[§5]{suslinStablyFreeModules1977}. The $m$-th Chern class induces a map $c_m:\GL\rightarrow\Kr(\Kbf_m^\Mr,m-1)$ and the composite $c_m\circ S_n$ induces a class in the abelian group $[Q_{2n-1},\Kr(\Kbf_m^\Mr,m-1))]_{\Abb^1}=\widetilde{\Hr}^{m-1}(Q_{2n-1},\Kbf_m^\Mr)$. By \cite[Lemma 4.5]{asokAlgebraicVectorBundles2014}, the group $\widetilde{\Hr}^{m-1}(Q_{2n-1},\Kbf_m^\Mr)$ vanishes if $m\neq n$, and if $m=n$, then $\widetilde{\Hr}^{n-1}(Q_{2n-1},\Kbf_n^\Mr)=\Zb\kappa$ where $\kappa$ is the $k$-invariant map \[\kappa:Q_{2n-1}\xrightarrow{\kappa_{n-1}}\Kr(\pi_{n-1}^{\Abb^1}(Q_{2n-1})=\Kbf_n^\MW,n-1)\to\Kr(\Kbf_n^\Mr,n-1).\] Now $c_n\circ S_n=(n-1)!\kappa$ by \cite[Remark 6.10]{duEnumeratingNonStableVector2022}. This yields the formula \[\Delta_d(\xi)\circ(S_{d+1})_*(\mu)=d!\kappa_*(\mu)\] for every $\mu\in[X_+,Q_{2d+1}]_{\Abb^1}$. To conclude that $\Delta_d(\xi)$ is surjective, it now suffices to prove the following statements.
\begin{itemize}
    \item \emph{The map $\kappa_*:[X_+,Q_{2d+1}]_{\Abb^1}\to[X_+,\Kr(\Kbf_{d+1}^\Mr,d)]_{\Abb^1}=\Hr^d(X,\Kbf_{d+1}^\Mr)$ is surjective.} This follows from \cite[Proposition 1.2.6]{asokEulerClassGroups2021}, which shows that the Hurewicz map \[(\kappa_d)_*:[X_+,Q_{2d+1}]_{\Abb^1}\to[X_+,\Kr(\pi_{d}^{\Abb^1}(Q_{2d+1}),d)]_{\Abb^1}=\Hr^d(X,\Kbf_{d+1}^\MW)\] is surjective, combined with Lemma \ref{lem:right_exactness_hd} which guarantees that the epimorphism $\Kbf_{d+1}^\MW\to\Kbf_{d+1}^\Mr$ of sheaves induces an epimorphism $\Hr^d(X,\Kbf_{d+1}^\MW)\to\Hr^d(X,\Kbf_{d+1}^\Mr)$ of groups.
    \item \emph{The group $\Hr^d(X,\Kbf_{d+1}^\Mr)$ is $d!$-divisible.} Since the hypothesis $(*)$ is satisfied, this follows from Proposition \ref{prop:divisibility_Hd-1Kd_real}.
\end{itemize}
This completes the proof.
\end{proof}

We thus obtain the following result, which generalises  \cite[Theorems 4.11 and 4.15]{dasOrbitSpacesUnimodular2018} when $X(\Rb)$ has no compact connected component and the smooth case of \cite[Theorem 2.8]{banerjeeZeroCyclesMennicke2025}.

\begin{cor}
Suppose that the hypothesis $(*)$ is satisfied. Then cancellation holds in corank $0$.
\end{cor}

\subsection{Corank $1$}

We now move to the case of corank $1$ bundles, which is considerably more subtle. We let $d\geq 3$ and we consider the fibre $F_{d-1}$ of the map $\pi_{d-1}:\BGL_{d-1}\to\BGL$. This fibre is $\Abb^1$-$(d-2)$-connected so the first two non-trivial stages $E_{d-1}$ and $E_d$ sit in fibre product diagrams of the form
\begin{center}
\begin{tikzcd}
E_{d-1} \arrow[r] \arrow[d,swap,"p_{d-1}"] & \Br\mathbb{G}_m \arrow[d] \\
\BGL \arrow[r]              & \Kr^{\mathbb{G}_m}(\pi_{d-1}^{\Abb^1}(F_{d-1}),d)
\end{tikzcd}
\begin{tikzcd}
E_{d} \arrow[r] \arrow[d,swap,"q_{d}"] & \Br\mathbb{G}_m \arrow[d] \\
E_{d-1} \arrow[r]            & \Kr^{\mathbb{G}_m}(\pi_{d}^{\Abb^1}(F_{d-1}),d+1)
\end{tikzcd}
\end{center}

We first note that the analogue of Lemma \ref{lem:bijectivity_corank_0} also holds in this context.

\begin{lem}\label{lem:bijectivity_corank_1}
The map $i_d:\BGL_{d-1}\to E_d$ determined from the Moore--Postnikov tower of $\pi_{d-1}$ induces a bijection $(i_d)_*:[X_+,\BGL_{d-1}]_{\Abb^1}\xrightarrow{\cong}[X_+,E_d]_{\Abb^1}$.
\end{lem}

\begin{proof}
The proof is the same as was provided for Lemma \ref{lem:bijectivity_corank_0}.
\end{proof}


\begin{lem}\label{lem:corank_1_first_stage_bijective}
Let $X$ be a smooth affine $\Rb$-scheme of dimension $d$ satisfying the hypothesis $(**')$. Then the map $(p_{d-1})_*:[X_+,E_{d-1}]_{\Abb^1}\to[X_+,\BGL]_{\Abb^1}$ is injective.
\end{lem}

\begin{proof}
Following the proof of Lemma \ref{lem:injectivity_corank_0}, we see that it suffices to prove the following statements.
\begin{itemize}
    \item \emph{The $k$-invariant $\kappa_{d-1}:Q_{2d-1}\rightarrow\Kr(\Kbf_d^\MW,d-1)$ induces a surjective homomorphism} \[[X_+,Q_{2d-1}]_{\Abb^1}\xrightarrow{(\kappa_{d-1})_*}\Hr^{d-1}(X,\Kbf_d^\MW)\to\Hr^{d-1}(X,\Kbf_d^\Mr).\] The first homomorphism is surjective by \cite[Proposition 1.2.6]{asokEulerClassGroups2021}. The cokernel of the homomorphism $\Hr^{d-1}(X,\Kbf_d^\MW)\to\Hr^{d-1}(X,\Kbf_d^\Mr)$ is a subgroup of $\Hr^{d}(X,\Ibf^d)$ which is isomorphic to the group $\Hr^d(X(\Rb),\Zb)$ by \cite[Theorem 3.5]{lerbetImageHigherSignature2026} and thus vanishes by the hypothesis $(**')$ (see Remark \ref{rem:relevance_cohomological_assumptions}).
    \item \emph{The group $\Hr^{d-1}(X,\Kbf_d^\Mr)$ is $(d-1)!$-divisible.} This follow from Proposition \ref{prop:divisibility_Hd-1Kd_real} since the hypothesis $(**')$ is satisfied.
\end{itemize}
The proof is therefore complete.
\end{proof}

The next injectivity statement is more delicate. We first use the functoriality of Moore--Postnikov towers, which yields a commutative ladder of motivic spaces:
\begin{equation}\label{eq:morphism_of_moore_postnikov_towers}
\begin{tikzcd}
E_d \arrow[r,"q_d"] \arrow[d] & E_{d-1} \arrow[r] \arrow[d,"p_{d-1}"] & \Kr^{\Gm}(\pi_d^{\Abb^1}(F_{d-1}),d+1) \arrow[d] \\
\widetilde{E}_d \arrow[r,swap,"\widetilde{p}_d"] & \BGL \arrow[r] & \Kr^{\Gm}(\pi_d^{\Abb^1}(F_d),d+1)
\end{tikzcd}
\end{equation}

\begin{lem}\label{lem:corank_one_trivial_fibre_second_stage}
Let $X$ be a smooth affine $\Rb$-scheme of dimension $d$ and let $\xi:X_+\to E_d$ be a map. Denote by $\Lc$ the line bundle classified by the composite \[X_+\xrightarrow{\xi} E_d\xrightarrow{p_d}\BGL\to\Br\mathbb{G}_m\] with the first $k$-invariant and assume that $(**_\Lc)$ holds. Then the fibre of the map $(q_d)_*:[X_+,E_d]_{\Abb^1}\to[X_+,E_{d-1}]_{\Abb^1}$ over $(q_d)_*\xi$ coincides with $\{\xi\}$.
\end{lem}

\begin{proof}
The ladder in (\ref{eq:morphism_of_moore_postnikov_towers}) induces a commutative ladder
\begin{center}
\begin{tikzcd}
\pi_0(\Omega_{(q_d)_*\xi}\Map(X_+,E_{d-1})) \arrow[r,"\partial"] \arrow[d,swap,"\protect{(p_{d-1})_*}"] & \Hr^d(X,\pi_{d}^{\Abb^1}(F_{d-1})(\Lc)) \arrow[r] \arrow[d] & \protect{[X_+,E_d]_{\Abb^1}} \arrow[r,"(q_d)_*"] \arrow[d] & \protect{[X_+,E_{d-1}]_{\Abb^1}} \arrow[d,"\protect{(p_{d-1})_*}"] \\
\pi_0(\Omega_{(p_{d-1}q_d)_*\xi}\Map(X_+,\BGL)) \arrow[r,swap,"\widetilde{\partial}"] & \Hr^d(X,\pi_{d}^{\Abb^1}(F_{d})(\Lc)) \arrow[r] & \protect{[X_+,\widetilde{E}_d]_{\Abb^1}} \arrow[r] & \protect{[X_+,\BGL]_{\Abb^1}} 
\end{tikzcd}
\end{center}
To prove that $(q_d)_*^{-1}((q_d)_*\xi)=\{\xi\}$, it is then sufficient to show that $\partial$ is surjective.

The vertical map $\Hr^d(X,\pi_d^{\Abb^1}(F_{d-1})(\Lc))\to\Hr^d(X,\pi_d^{\Abb^1}(F_{d})(\Lc))$ is an isomorphism by Lemma \ref{lem:iso_homotopy_sheaves_of_fibres}, and $\Hr^d(X,\pi_d^{\Abb^1}(F_{d})(\Lc))\cong\Hr^d(X,\Kbf_{d+1}^\Mr)$ as noted in the proof of Lemma \ref{lem:injectivity_corank_0}. Translation along $(p_{d-1}q_d)_*\xi$ induces maps \[t:[X_+,\GL]_{\Abb^1}=\pi_0(\Omega_*\Map(X_+,\BGL))\to\pi_0(\Omega_{(p_{d-1}q_d)_*\xi}\Map(X_+,\BGL))\] and \[\Delta_d(\xi):[X,\GL]_{\Abb^1}\xrightarrow{t}\pi_0(\Omega_{(p_{d-1}q_d)_*\xi}\Map(X_+,\BGL))\to\Hr^d(X,\Kbf_{d+1}^\Mr)\] fitting in a commutative diagram
\begin{equation}\label{eq:second_injectivity_corank_1}
\begin{tikzcd}
& \pi_0(\Omega_{(q_d)_*\xi}\Map(X_+,E_{d-1})) \arrow[r,"\partial"] \arrow[d,swap,"\protect{(p_{d-1})_*}"] &  \Hr^d(X,\pi_{d}^{\Abb^1}(F_{d-1})(\Lc)) \arrow[d,"\cong"] \\
& \pi_0(\Omega_{(p_{d-1}q_d)_*\xi}\Map(X_+,\BGL)) \arrow[r,"\widetilde{\partial}"] & \Hr^d(X,\Kbf_{d+1}^\Mr) \\
\protect{[X_+,Q_{2d+1}]_{\Abb^1}} \arrow[r,swap,"(S_{d+1})_*"] \arrow[ru,"\protect{t(S_{d+1})_*}"] & \protect{[X,\GL]_{\Abb^1}} \arrow[u,"t"] \arrow[ru,swap,"\protect{\Delta_d(\xi)}"] & 
\end{tikzcd}
\end{equation}
where $S_{d+1}:Q_{2d+1}\to\GL$ is the morphism of spaces induced by Suslin matrices. As we saw in the proof of Lemma \ref{lem:injectivity_corank_0}, the map $\Delta_d(\xi)\circ(S_{d+1})_*$ is surjective. To prove that $\partial$ is surjective, it then suffices to prove that the map $t\circ(S_{d+1})_*$ factors through $(p_{d-1})_*$ into a map $f:[X_+,Q_{2d+1}]_{\Abb^1}\to\pi_0(\Omega_{(q_d)_*\xi}\Map(X_+,E_{d-1}))$. Indeed, assuming that $t\circ(S_{d+1})_*=(p_{d-1})_*\circ f$, let $\alpha\in\Hr^d(X,\pi_d^{\Abb^1}(F_{d-1})(\Lc))$: we wish to show that $\alpha=\partial\alpha'$ for some $\alpha'\in\pi_0(\Omega_{(q_d)_*\xi}\Map(X_+,E_{d-1}))$. Let $\beta$ denote the image of $\alpha$ in $\Hr^d(X,\Kbf_{d+1}^\Mr)$. Since $\Delta_d(\xi)\circ(S_{d+1})_*$ is surjective, there exists $\mu\in[X_+,Q_{2d+1}]_{\Abb^1}$ such that $\Delta_d(\xi)\circ(S_{d+1})_*(\mu)=\beta$. Set $\alpha'=f(\mu)$. Then by inspection of (\ref{eq:second_injectivity_corank_1}), we see that the elements $\partial\alpha'$ and $\alpha$ of $\Hr^d(X,\pi_{d-1}(F_{d-1})(\Lc))$ have the same image in $\Hr^d(X,\Kbf_{d+1}^\Mr)$, namely $\beta$. Since the homomorphism $\Hr^d(X,\pi_{d}^{\Abb^1}(F_{d-1})(\Lc))\to\Hr^d(X,\Kbf_{d+1}^\Mr)$ is injective, this shows that $\partial\alpha'=\alpha$.

It remains to see that $t\circ(S_{d+1})_*$ factors through a map $f$ as indicated in the previous paragraph. To do so, we consider the pullback diagram
\begin{center}
\begin{tikzcd}
E_{d-1} \arrow[r] \arrow[d,swap,"p_{d-1}"] & \Br\mathbb{G}_m \arrow[d] \\
\BGL \arrow[r]                        & \Kr^{\mathbf{G}_m}(\pi_{d-1}^{\Abb^1}(F_{d-1}),d)
\end{tikzcd}
\end{center}
of motivic spaces. In view of the isomorphism $\Hr^{d-1}(X,\pi_{d-1}^{\Abb^1}(F_{d-1})(\Lc))\cong\Hr^{d-1}(X,\Kbf_d^\Mr)$ of Lemma \ref{lem:cohomology_homotopy_sheaves_milnor}, it induces a commutative diagram
\begin{center}
\begin{tikzcd}
\pi_0(\Omega_{(q_d)_*\xi}\Map(X_+,E_{d-1})) \arrow[r,"(p_{d-1})_*"] & \pi_0(\Omega_{(p_{d-1}q_d)_*\xi}\Map(X_+,\BGL)) \arrow[r,"\delta'"] & \Hr^{d-1}(X,\Kbf_d^\Mr) \\
\protect{[X,Q_{2d+1}]_{\Abb^1}} \arrow[r,swap,"(S_{d+1})_*"] & \protect{[X_+,\GL]_{\Abb^1}} \arrow[u,"t"] \arrow[ru,swap,"\protect{\Delta_{d-1}(\xi)}"] & 
\end{tikzcd}
\end{center}
whose top row is exact in the usual sense.  In particular, the fibre of $\delta'$ over $0\in\Hr^{d-1}(X,\Kbf_d^\Mr)$ is the image of $(p_{d-1})_*$ so to show that $t\circ(S_{d+1})_*$ factors through $(p_{d-1})_*$, it suffices to prove that the composite \[[X_+,Q_{2d+1}]_{\Abb^1}\xrightarrow{(S_{d+1})_*}[X,\GL]_{\Abb^1}\xrightarrow{\Delta_{d-1}(\xi)}\Hr^{d-1}(X,\Kbf_d^\Mr)\] is trivial. For every $\mu\in[X_+,Q_{2d+1}]_{\Abb^1}$, we have \[\Delta_{d-1}(\xi)\circ(S_{d+1})_*(\mu)=c_d((S_{d+1})_*\mu)+\sum_{j=1}^{d-1}c_j((S_{d+1})_*\mu)c_{d-j}(\xi).\] But $c_m((S_{d+1})_*\mu)=0$ for $m\neq d+1$ by \cite[Proposition 6.9, Remark 6.10]{duEnumeratingNonStableVector2022} again. Therefore $\Delta_{d-1}(\xi)\circ(S_{d+1})_*$ is trivial and the proof is complete.
\end{proof}

\begin{cor}\label{cor:cancellability_corank_1}
Let $X$ be a smooth affine $\Rb$-variety of dimension $d$ and let $\Esc$ be a rank $d-1$ vector bundle on $X$. Suppose that the hypothesis $(**_{\det\Esc})$ is satisfied. Then $\Esc$ is cancellative. In particular, if $(**_\Lc)$ holds for any line bundle $\Lc$ on $X$, then every rank $d-1$ vector bundle on $X$ is cancellative.
\end{cor}

\begin{proof}
Let $\xi:X_+\to\BGL_{d-1}$ classify $\Esc$. As noted previously, the cancellability of $\Esc$ is equivalent to the assertion that the fibre of the map $(\pi_{d-1})_*:[X_+,\BGL_{d-1}]_{\Abb^1}\to[X_+,\BGL]_{\Abb^1}$ over $(\pi_{d-1})_*\xi$ is a singleton. This map factors as \[(\pi_{d-1})_*:[X_+,\BGL_{d-1}]_{\Abb^1}\xrightarrow{(i_d)_*}[X_+,E_d]_{\Abb^1}\xrightarrow{(q_d)_*}[X_+,E_{d-1}]_{\Abb^1}\xrightarrow{(p_{d-1})_*}[X,\BGL]_{\Abb^1}.\] In this diagram, the map $(i_d)_*$ is bijective by Lemma \ref{lem:bijectivity_corank_1}. By Remark \ref{rem:double_star_weaker}, since $(**_{\det\Esc})$ is satisfied, so is $(**')$, hence $(p_{d-1})_*$ is bijective by Lemma \ref{lem:corank_1_first_stage_bijective}. Finally, by Lemma \ref{lem:corank_one_trivial_fibre_second_stage}, the fibre of $(q_d)_*$ over $(q_d)_*(i_d)_*\xi$ is the singleton reduced to $(i_d)_*\xi$. This establishes the first claim, and the second follows immediately.
\end{proof}

\begin{ex}\label{ex:partial_cancellability_corank_1}
Let $X\subseteq\mathbb{P}_\Rb^3$ be the affine variety constructed in Example \ref{exe:different_satisfaction_double_star}. Recall that $X(\Rb)$ has no compact connected component and that $\Hr^2(X(\Rb),\Zb(\Osc(1)_{|X}))=0$. Then by Corollary \ref{cor:cancellability_corank_1}, any rank $2$ bundle on $X$ with determinant $\Osc(1)_{|X}$ (in $\Pic(X)/2$) is cancellative. However, since $\Hr^2(X(\Rb),\Zb)=\Zb/2$ is nonzero, our methods do not allow us to show that rank $2$ bundles whose determinant is a square are cancellative. For instance, we do not know whether stably free rank $2$ bundles on $X$ are free.
\end{ex}

\begin{rem}
Similar methods may be used to study prestabilisation for (special) $K$-theory. Recall that if $A$ is a commutative ring, there is a natural stabilisation group homomorphism $\varphi_n:\mathrm{SL}_n(A)/\Er_n(A)\to\mathrm{SK}_1(A)$ for every $n\geq 3$, where $\mathrm{SK}_1(A)$ is the kernel of the determinant map from the $\Kr$-theory group $\Kr_1(A)$ to $A^\times$ and $\Er_n(A)$ is the subgroup of elementary matrices in $\mathrm{SL}_n(A)$. If $A$ is noetherian of dimension $d$, then $\varphi_n$ is bijective for every $n\geq d+2$ and surjective if $n\geq d+1$. Prestabilisation occurs when $\varphi_n$ is injective or surjective for ``unexpected'' values of $n$, namely for some $n\leq d+1$. In \cite{banerjeeZeroCyclesMennicke2025}, the first author proved that $\varphi_{d+1}$ is injective if $A$ is smooth and has no real maximal ideal.\footnote{In \cite{banerjeeZeroCyclesMennicke2025}, the precise condition on the rings under consideration is that one of the following hypotheses holds: (i) the ring has no real maximal ideal, or (ii) the intersection of these ideals has height $\geq 1$. But if $A$ is a smooth affine algebra of dimension $d$ over $\Rb$, then under (ii), the set of real maximal ideals of $A$ underlies a smooth manifold of dimension $\leq d-1$. Since this smooth manifold is also of dimension $d$, it must be empty so $A$ in fact satisfies (i) (see also \cite[Example 3.16]{lerbetCohomologicalClassificationVector2026}).} This was recently extended to the map $\varphi_d$ in joint work of the first author with K. Chakraborty \cite[Lemma 9.2]{banerjeeImprovedInjectiveStability2026}. Obstruction theory as explained in this section can be used to prove much finer results. Indeed, the same method shows that if $A$ is smooth over $\Rb$, then setting $X=\Spec A$, under the hypothesis $(*)$, the map $\varphi_{d+1}$ is injective due to the divisibility of $\Hr^d(X,\Kbf_{d+2}^\Mr)$, and if $(**')$ holds, then $\varphi_d$ is injective due to the divisibility of $\Hr^{d-1}(X,\Kbf_{d+1}^\Mr)$.
\end{rem}

\subsection{The singular case}

\begin{thm}
	\label{cancellation}
	Let $X = \Spec (A)$ be a normal affine $\mathbb{R}$-variety of dimension $d$. Assume that either $X(\mathbb{R}) = \emptyset$, or the intersection of all real maximal ideals of $A$ has height at least $2$. If $d = 3$, suppose in addition that $X$ is smooth. Then every stably free $A$-module of rank $d-1$ is free.
	
\end{thm}

\proof Let $P$ be a stably free $A$-module of rank $d-1$. Since the result is clear for $d=2$, we assume that $d \geq 3$. By \cite[Theorem 2.8]{banerjeeZeroCyclesMennicke2025}, we have $P \oplus A \simeq A^d$. Hence, $P$ is defined by a unimodular row $v = (v_1, \ldots, v_d) \in \Um_d(A)$. In view of \cite[Theorem 2]{suslinStablyFreeModules1977}, it suffices to show that there exists a unimodular row $(u_1, \ldots, u_d) \in \Um_d(A)$ such that
$v = (u_1^{(d-1)!}, u_2, \ldots, u_d) \text{ in } \Um_d(A)/\mathrm{E}_d(A).$

Suppose that $d \geq 4$. Let $\mathcal{J}_1$ be the ideal of the singular locus of $A$, and let $\mathcal{J}_2$ be the intersection of all real maximal ideals whenever $X(\mathbb{R}) \neq \emptyset$; otherwise, set $\mathcal{J}_2 = A$. Since $A$ is normal, one has $\operatorname{ht}(\mathcal{J}_1) \geq 2$. We set $\mathcal{J} = \mathcal{J}_1 \mathcal{J}_2$; then $\dim(A/\mathcal{J}) \leq d-2$. It follows from Bass's cancellation theorem \cite{bassAlgebraicKtheory1968} that $\mathrm{Um}_d(A/\mathcal{J}) = e_1 \mathrm{E}_d(A/\mathcal{J})$, where $e_1$ is the vector $(1,0,\dots,0)$. We can therefore assume, performing elementary operations if necessary, that $v_d - 1 \in \mathcal{J}$ and $v_1, \ldots, v_{d-1} \in \mathcal{J}$. Using now a Bertini-type theorem due to Swan \cite[Theorem 1.5]{swanCancellationTheoremProjective1974}, we may further alter $v$ by elementary transformations, if necessary, to assume that $B := A/\langle v_4, \ldots, v_d\rangle $ is a three-dimensional smooth $\mathbb{R}$-variety satisfying $Y(\mathbb{R}) = \emptyset$, where $Y = \Spec(B)$.

Let “bar” denote reduction modulo $\langle v_4, \ldots, v_d \rangle$. Then there is a well-defined map
$$\Um_3(B)/\mathrm{E}_3(B) \to \Um_d(A)/\mathrm{E}_d(A)$$
such that $(\overline{x}, \overline{y}, \overline{z}) \mapsto (x, y, z, v_4, \ldots, v_d)$, where $(\overline{x}, \overline{y}, \overline{z}) \in \Um_3(B)$ and $x, y, z \in A$ are lifts of $\overline{x}, \overline{y}, \overline{z}$, respectively. Therefore, it is enough to show that there exists $(\overline{a}, \overline{b}, \overline{c}) \in \Um_3(B)$ such that $(\overline{v}_1, \overline{v}_2, \overline{v}_3) = (\overline{a}^{(d-1)!}, \overline{b}, \overline{c})$ in $\Um_3(B)/\mathrm{E}_3(B)$. We are thus
reduced to the case where $X=\Spec(A)$ is a \emph{smooth} affine threefold over $\mathbb{R}$ such that $X(\mathbb R)=\emptyset$. For the remaining part of the proof, we assume that $A$ is a smooth affine $\Rb$-algebra of dimension $3$ with spectrum $X$ such that $X(\Rb)$ is empty.

By Corollary \ref{cor:mennicke_bijective} (which applies as $X(\Rb)$ is empty), the map $\psi_2:\Um_3(A)/\Er_3(A)\to\Hr^2(X,\Kbf_3^\Mr)$ is bijective and therefore induces a group structure on $\Um_3(A)/\Er_3(A)$; we write this group structure \emph{multiplicatively}. The group $\Um_3(A)/\Er_3(A)$ is then divisible by Proposition \ref{prop:divisibility_Hd-1Kd_real} so there exists $(\overline{a}, \overline{b}, \overline{c})\in\Um_3(A)$ such that $(\overline{v}_1, \overline{v}_2, \overline{v}_3)=(\overline{a}, \overline{b}, \overline{c})^{(d-1)!}$. But by Proposition \ref{Milnor K-cohomology is a Mennicke symbol}, the map $\psi_2$ is also a Mennicke symbol. In particular, the equality $(\overline{a}, \overline{b}, \overline{c})^{(d-1)!}=(\overline{a}^{(d-1)!}, \overline{b}, \overline{c})$ holds in $\Um_3(A)/\Er_3(A)$ by (2) in Definition \ref{weakMennickeSymbol}. It follows that \[(\overline{v}_1, \overline{v}_2, \overline{v}_3) =(\overline{a}, \overline{b}, \overline{c})^{(d-1)!} =(\overline{a}^{(d-1)!}, \overline{b}, \overline{c})\text{ in } \Um_3(A)/\mathrm{E}_3(A).\]
This concludes the proof.\qed

\begin{rem}\label{Fasel'sRemark} In fact, the proof of the above theorem actually gives a bit more. Let $J$ be the ideal defining the singular locus of $A$, and let $v\in \Um_d(A)$ be a
	unimodular row of length $d$ which is in the elementary orbit of $e_1$ in $\frac{\Um_d (R/J)}{\mathrm{E}_d (R/J)}$. Then the stably free module associated to $v$ is free.
	
\end{rem} 
\section{Chow--Witt groups and Euler class groups}

Let $k$ be a field, let $n\geq 1$ be an integer and let $Q_{2n}$ be the smooth affine quadric whose coordinate ring is 
\[
k[Q_{2n}]=k[x_1,\ldots,x_n,y_1,\ldots,y_n,z]/(\sum_{i=1}^nx_iy_i-z(1-z)).
\]
The $k$-variety is equivalent to the space $(\pone)^{\wedge n}$ by \cite{asokSmoothModelsMotivic2017}, and consequently
\[
\piaone_i(Q_{2n})=\begin{cases} 0 & \text{ if $i<n$.} \\ \KMW_n & \text{ if $i=n$.}\end{cases}
\]
In particular, the quadric $Q_{2n}$ is $\A^1$-simply connected if $n\geq 2$, and we thus assume that this condition holds in the sequel. The $\kappa$-invariant 
\[
Q_{2n}\to \mathrm{K}(\KMW_{n},n) 
\]
corresponds (up to a unit in $\mathrm{GW}(k)$) to an element $\CHW^n(Q_{2n})$ described in \cite[Lemma 1.1.6]{asokEulerClassGroups2021}. Using the morphism of sheaves $\KMW_n\to \KM_n$, we obtain a composite map 
\[
Q_{2n}\to \mathrm{K}(\KMW_{n},n)\to \mathrm{K}(\KM_{n},n)
\]
corresponding to a fundamental class $\theta_n\in \CH^n(Q_{2n})$, which is fact the class of the cycle $Z_n:=\{x_1=\cdots=x_n=z=0\}$ of codimension $n$.

The following result will be of interest for us in this section (\cite[Proposition 1.2.2]{asokEulerClassGroups2021}).

\begin{prop}
Let $n\geq 3$ and let $X$ be a smooth scheme of dimension $\leq 2n-2$ over a perfect field $k$, then the stabilization map
\[
[X_+,Q_{2n}]_{\A^1}\to [\Sigma_{S^1}^\infty X_+,\Sigma_{S^1}^\infty Q_{2n}]_{\A^1}
\]
is a bijection, functorial in the first input (for smooth schemes of dimension $\leq 2n-2$). Consequently, the left-hand side is endowed with the structure of an abelian group.
\end{prop}

Under the conditions of the theorem, the $\kappa$-invariant yields a map 
\[
[X_+,Q_{2n}]_{\A^1}\to \Hr^n(X,\KMW)=\CHW^n(X)
\]
which turns out to be a group homomorphism (\cite[p. 2791]{asokEulerClassGroups2021}). On the other hand, the abelian group $[X,Q_{2n}]_{\A^1}$ can be expressed in terms of generators and relations in case $X$ is affine; more precisely, it is naturally isomorphic to the so-called \emph{Euler class group} $\mathrm{E}^{n}(X)$ defined in \cite{bhatwadekarEulerClassesStably2002}.  

\subsection*{Weak Euler class groups}

 Let $R$ be a commutative noetherian ring of dimension $d\ge 3$. In this subsection, we follow the set up of \cite{bhatwadekarEulerClassGroup2000} to define the $n$-th weak Euler class group ${\Er^n_0}(R)$ of $R$, where $2n \ge d + 3$.
 
Let $G$ be the free abelian group generated by symbols $(I)$, where $I \subset R$ is an ideal satisfying the following conditions:
\begin{enumerate}
\item $\Spec(R/I)$ is connected.
\item $\operatorname{ht}({I})=n=\mu(I/I^2)$, where for an $R$-module $M$, we let $\mu(M)$ denote the smallest integer $n$ such that there exists a generating family of $M$ with $n$ elements.
\end{enumerate}

Let $I \subset R$ be an ideal of height $n$ such that 
\[
I = \bigcap_{i=1}^{s} {I}_i,
\]where ${I}_i$ are proper, pairwise comaximal ideals of height $n$, and $\Spec(R/I_i)$ are connected. It follows from \cite[Lemma 4.1]{bhatwadekarEulerClassGroup2000} that such a decomposition is unique. We shall say that the ${I}_i$ are the \emph{connected components} of $I$.

Let $I \subset R$ be an ideal of height $n$ with connected components ${I}_i$ for $i=1,\dots, s$, satisfying further that $\mu(I/I^2) = n$. 
The Chinese remainder theorem shows that $\mu({I}_i/{I}_i^2) = n$ for each $i = 1, \dots, s$ and we set
\[
(I) := \sum_{t=i}^s ({I}_i) \in G.
\]

We now consider the subgroup $H\subset G$ generated by symbols $(J)$, where $\mu(J)=n$ and we define the $n$-th weak Euler class group as 
\[
\mathrm{E}^n_0(R):=G/H.
\]
If $X=\Spec(R)$, we sometimes slightly abuse notation and write $\mathrm{E}^n_0(X)$ in place of $\mathrm{E}^n_0(R)$.

\begin{lem}\label{lem:ontoECGtowECG}
There exists a natural surjection $\gamma^n\colon \mathrm{E}^{n}(R) \longtwoheadrightarrow \mathrm{E}^{n}_0(R)$. 
\end{lem}

\begin{proof} 
Recall from \cite[p. 147]{bhatwadekarEulerClassesStably2002} (or \cite[Definition 3.1.5]{asokEulerClassGroups2021}) that the group $\mathrm{E}^{n}(R)$ is obtained as the quotient of the free abelian group on generators of the form $(I,\omega_I)$, where $I$ is an ideal as above and $\omega_I\colon (R/I)^n\to I/I^2$ is a (or rather an equivalence class of) surjection, by the subgroup of complete intersections. With this in mind, it is clear that the map $(I,\omega_I)\mapsto I$ induces a well defined homomorphism $\gamma^n\colon \mathrm{E}^{n}(R) \to \mathrm{E}^{n}_0(R)$. To prove that this map is onto, it suffices to observe that if $(J)\in \mathrm{E}^n_0(R)$ we may choose a local orientation $\omega_J$ of $J$, obtaining a pair $(J,\omega_J)$ in $\mathrm{E}^{n}(R)$. 
\end{proof}

\begin{prop}
  Let $X=\Spec(R)$ be a smooth affine $d$-fold over a field $k$ and $2n\ge d+3$. Then, the composite
\[
\Er^n(R)\to \CHW^{n}(X)\to  \CH^{n}(X)
\]
factors through the map $\gamma^n\colon \Er^n(R)\to \Er^n_0(R)$. Consequently, there is a commutative diagram
\[
\xymatrix{\Er^{n}(R)\ar[r]\ar[d]_-{\gamma^n} & \CHW^{n}(X)\ar[d] \\
\Er_0^{n}(R)\ar[r]& \CH^{n}(X).}
\]
\end{prop}

\begin{proof}
As explained above, the composite $\Er^n(R)\to \CHW^n(X)\to \CH^n(X)$ is induced by the composite
\[
Q_{2n}\xrightarrow{\kappa} \mathrm{K}(\KMW_n,n)\to \mathrm{K}(\KM_n,n)
\]
corresponding to the class of $Z_n$ in $\CH^n(Q_{2n})$. Any element $(I,\omega_I)$ in $\Er^n(R)$ corresponds to a homotopy class of map $f\colon X\to Q_{2n}$, and its image in $\CH^n(X)$ is just $f^*(Z_n)$. The result now follows from the observation that $f^*(Z_n)$ depends only on $I$ in the pair $(I,\omega_I)$.
\end{proof}

\begin{prop}
Let $X$ be a smooth affine real $d$-fold, with $d\geq 4$. If $X$ satisfies the hypothesis $(*)$, the Hurewicz homomorphism 
\[
[X_+,Q_{2d-2}]_{\A^1}\to \CHW^{d-1}(X)
\]
is an isomorphism. 
\end{prop}

\begin{proof}
The Hurewicz homomorphism is obtained by looking at the Moore--Postnikov tower of the map $Q_{2d-2}\to \Spec(\real)$. Since $X$ is of dimension $d$, the relevant part of the tower takes the form of a fiber sequence 
\[
(\mathrm{K}(\piaone_{d}(Q_{2d-2}),d)\to Q_{2d-2}\xrightarrow{\kappa}\mathrm{K}(\piaone_{d-1}(Q_{2d-2}),d-1)\to \mathrm{K}(\piaone_{d}(Q_{2d-2}),d+1).
\]
For dimension reasons, we have
\[
[X_+,\mathrm{K}(\piaone_{d}(Q_{2d-2}),d+1)]_{\aone}=\Hr^{d+1}(X,\piaone_{d}(Q_{2d-2}))=0
\]
and we are left to show that $[X_+,\mathrm{K}(\piaone_{d}(Q_{2d-2}),d)]_{\aone}=\Hr^d(X,\pi_{d}^{\Abb^1}(Q_{2d-2}))=0$ as well to conclude. Since $d\geq 4$, Theorem \ref{thm:homotopy_sheaves_of_motivic_spheres} shows that the suspension homomorphism $\pi_{d-1}^{\Abb^1}(\Abb^{d-1}\smallsetminus\{0\})\to \pi_{d}^{\Abb^1}(Q_{2d-2})$ is onto. In view of Lemma \ref{lem:right_exactness_hd}, we are left to show that $\Hr^d(X,\pi_{d-1}^{\Abb^1}(\Abb^{d-1}\smallsetminus\{0\}))=0$. If $d\geq 5$, we may use the exact sequence
\[
0\rightarrow\Kbf_{d+1}^\Mr/24\rightarrow\pi_{d-1}^{\Abb^1}(\Abb^{d-1}\smallsetminus \{0\})\rightarrow\GW_{d}^{d-1}
\]
that becomes exact after $d-3$ contractions, yielding an exact sequence
\[
\Hr^d(X,\Kbf_{d+1}^\Mr/24)\to \Hr^d(X,\pi_{d-1}^{\Abb^1}(\Abb^{d-1}\smallsetminus \{0\})) \to \Hr^d(X,\GW_{d}^{d-1})\to 0.
\]
By \cite[Theorem 3.7.1]{asokSplittingVectorBundles2014}, the group $\Hr^d(X,\GW_{d}^{d-1})$ is a quotient of $\Hr^d(X,\Kbf_d^\Mr)/2$, hence vanishes under the hypothesis $(*)$ by Proposition \ref{prop:divisibility_Hd-1Kd_real}, and the same holds for $\Hr^d(X,\Kbf_{d+1}^\Mr/24)$.

If now $d=4$, we use instead the exact sequence
\[
 0\to\Fbf_5\rightarrow\pi_3^{\Abb^1}(\Abb^3\smallsetminus \{0\})\to\GW_4^3\to 0
\] 
and the same arguments as above show that it suffices to prove that $\Hr^4(X,\Fbf_5)=0$ to conclude. Since we have an exact sequence  \[\Ibf^6\rightarrow\Tbf_5\rightarrow\Sbf_5\rightarrow 0\] and the sheaf $\Sbf_5$ is a quotient of $\Kbf_5^\Mr/24$, we are reduced to showing that $\Hr^4(X,\Ibf^6)=0$. Now
\[
\Hr^4(X,\Ibf^6)=\Hr^4(X(\Rb),\Z)=0
\] 
by Remark \ref{rem:relevance_cohomological_assumptions}.
\end{proof}

This result allows to prove the following theorem.

\begin{thm}
Let $X$ be a smooth affine real $d$-fold, with $d\geq 4$. If the hypothesis $(**)$ is satisfied, then, all homomorphisms in the commutative diagram
\[
\xymatrix{\Er^{d-1}(X)\ar[r]\ar[d] & \CHW^{d-1}(X)\ar[d] \\
\Er_0^{d-1}(X)\ar[r]& \CH^{d-1}(X)}
\]
are isomorphisms.
\end{thm}

\begin{proof}
By the previous proposition, the map $\Er^{d-1}(X)\to \CHW^{d-1}(X)$ is an isomorphism. 

Next, we use the exact sequence of sheaves
\[
0\to \mathbf{I}^d\to \KMW_{d-1}\to \KM_{d-1}\to 0
\]
to obtain an exact sequence
\[
\Hr^{d-1}(X,\mathbf{I}^d)\to \CHW^{d-1}(X)\to \CH^{d-1}(X)\to \Hr^{d}(X,\mathbf{I}^d). 
\]
Since $(**)$ holds, the right-hand and left-hand terms are trivial (see Remark \ref{rem:relevance_cohomological_assumptions}) and consequently the homomorphism $\CHW^{d-1}(X)\to \CH^{d-1}(X)$ is an isomorphism. Thus, the composite
\[
\Er^{d-1}(X)\to \Er^{d-1}_0(X)\to \CH^{d-1}(X)
\]
is an isomorphism and then the first map is injective as well.  On the other hand, $\Er^{d-1}(X)\to \Er^{d-1}_0(X)$ is also onto by construction. It is then an isomorphism, and thus the bottom horizontal homomorphism as well.
\end{proof}

{\begin{footnotesize}
\raggedright
\bibliographystyle{alpha}
\bibliography{splitting.bib}

@unpublished{banerjeeImprovedInjectiveStability2026,
  title = {Improved injective stability for relative $\mathrm{K_1Sp}$-groups},
  author = {Banerjee, Sourjya and Chakraborty, Kuntal},
  year = {2026},
  note = {{2604.09476}},
}

@article{asokMotivicspheres20,
	author = {Asok, Asok and Fasel, Jean and Williams, T. Ben},
	journal = {Invent. math.},
	number = {1},
	pages = {39-73},
	title = {{Motivic spheres and the image of the Suslin-Hurewicz map}},
	volume = {219},
	year = {2020}}

@article{asokAlgebraicVectorBundles2014,
	author = {Asok, Aravind and Fasel, Jean},
	journal = {J. Topology},
	number = {3},
	pages = {894--926},
	title = {Algebraic vector bundles on spheres},
	volume = {7},
	year = {2014}}

@article{asokCohomologicalClassificationVector2014,
	author = {Asok, Aravind and Fasel, Jean},
	journal = {Duke Math. J.},
	number = {14},
	pages = {2561--2601},
	title = {A cohomological classification of vector bundles on smooth affine threefolds},
	volume = {163},
	year = {2014}}

@article{asokSplittingVectorBundles2014,
	author = {Asok, Aravind and Fasel, Jean},
	journal = {J. Am. Math. Soc.},
	number = {4},
	pages = {1031--1062},
	title = {Splitting vector bundles outside the stable range and {$\mathbb{A}^1$}-homotopy sheaves of punctured affine spaces},
	volume = {28},
	year = {2014}}

@article{asokAffineRepresentabilityResults2017a,
	author = {Asok, Aravind and Hoyois, Marc and Wendt, Matthias},
	journal = {Duke Math. J.},
	number = {10},
	pages = {1923--1953},
	shorttitle = {Affine Representability Results in {{A1-homotopy}} Theory, {{I}}},
	title = {Affine representability results in {{$\mathbb{A}^1$-homotopy}} theory, {{I}}: {{Vector}} bundles},
	volume = {166},
	year = {2017}}

@article{asokSimplicialSuspensionSequence2017,
	author = {Asok, Aravind and Wickelgren, Kirsten and Williams, Ben},
	journal = {Geometry \& Topology},
	number = {4},
	pages = {2093--2160},
	publisher = {MSP},
	title = {The simplicial suspension sequence in $\mathbb{A}^1$-homotopy},
	volume = {21},
	year = 2017}

@article{asokSmoothModelsMotivic2017,
	author = {Asok, Aravind and Doran, Brent and Fasel, Jean},
	journal = {Int. Math. Res. Not.},
	number = {6},
	pages = {1890--1925},
	title = {Smooth models of motivic spheres and the clutching construction},
	volume = {2017},
	year = 2017}

@article{asokEulerClassGroups2021,
	author = {Asok, Aravind and Fasel, Jean},
	journal = {J. Eur. Math. Soc.},
	number = {8},
	pages = {2775--2822},
	title = {Euler class groups and motivic stable cohomotopy (with an appendix by {{Mrinal Kanti Das}})},
	volume = {24},
	year = {2021}}

@article{asokP1stabilizationUnstableMotivic,
	author = {Asok, Aravind and Bachmann, Tom and Hopkins, Michael J.},
	journal = {Ann. of Math.},
	title = {On {{$\mathbb{P}^1$-stabilization}} in unstable motivic homotopy theory},
	year = {to appear}}

@unpublished{asokSplittingVectorBundles2025,
	author = {Asok, Asok and Fasel, Jean and Lerbet, Samuel},
	note = {{arXiv:2511.15616}},
	title = {{Splitting vector bundles over real algebraic varieties}},
	year = {2025}}

@article{banerjeeZeroCyclesMennicke2025,
	author = {Banerjee, Sourjya},
	journal = {Int. Math. Res. Not.},
	number = {15},
	pages = {rnaf211},
	title = {Zero {{cycles}}, {{Mennicke symbols}}, and {{$K_1$-stability}} of {{certain real affine algebras}}},
	volume = {2025},
	year = 2025}

@book{bassAlgebraicKtheory1968,
	author = {Bass, Hyman},
	location = {New York (N.Y.) ; Amsterdam},
	number = {11},
	publisher = {W. A. Benjamin},
	series = {Mathematics Lecture Note Series},
	title = {Algebraic {{$K$-theory}}},
	year = {1968}}

@article{bhatwadekarEulerClassGroup2000,
	author = {Bhatwadekar, Shrikant M. and Sridharan, Raja},
	journal = {Compos. Math},
	number = {2},
	pages = {183--222},
	title = {The {{Euler class group}} of a {{Noetherian ring}}},
	volume = {122},
	year = {2000}}

@incollection{bhatwadekarEulerClassesStably2002,
	author = {Bhatwadekar, Shrikant M. and Sridharan, Raja},
	booktitle = {Algebra, {{arithmetic}} and {{geometry}}, {{part I}}, {{II}} ({{Mumbai}}, 2000)},
	editor = {Parimala, Raman},
	pages = {139--158},
	publisher = {Narosa Publ. House},
	series = {Tata {{Inst}}. {{Fund}}. {{Res}}. Stud. Math.},
	title = {On {{Euler}} classes and stably free projective modules},
	volume = {5},
	year = 2002}

@article{blochGerstensConjectureHomology1974,
	author = {Bloch, Spencer and Ogus, Arthur},
	journal = {Ann. sci. \'Ec. norm. sup.},
	number = {2},
	pages = {181--201},
	title = {Gersten's conjecture and the homology of schemes},
	volume = {7},
	year = {1974}}

@article{cadekCohomologyBON1999,
  title = {The cohomology of $\mathrm{BO}(n)$ with twisted integer coefficients},
  author = {{\v C}adek, Martin},
  year = {1999},
  journal = {Kyoto Journal of Mathematics},
  volume = {39},
  number = {2},
  pages = {277--286},
}

@article{colliot-theleneZerocyclesCohomologyReal1996,
	author = {{Colliot-Th{\'e}l{\`e}ne}, Jean-Louis and Scheiderer, Claus},
	journal = {Topology},
	number = {2},
	pages = {533--559},
	title = {Zero-cycles and cohomology on real algebraic varieties},
	volume = {35},
	year = {1996}}

@article{dasOrbitSpacesUnimodular2018,
	author = {Das, Mrinal K. and Tikader, Soumi and Zinna, Md. Ali},
	journal = {Invent. math.},
	number = {1},
	pages = {133--159},
	title = {Orbit spaces of unimodular rows over smooth real affine algebras},
	volume = {212},
	year = {2018}}

@article{duEnumeratingNonStableVector2022,
	author = {Du, Peng},
	journal = {Int. Math. Res. Not.},
	number = {19},
	pages = {14797--14864},
	title = {Enumerating {{non-stable vector bundles}}},
	volume = {2022},
	year = {2022}}

@article{faselRemarksOrbitSets2010a,
	author = {Fasel, Jean},
	journal = {Comment. Math. Helv.},
	number = {1},
	pages = {13--39},
	title = {Some remarks on orbit sets of unimodular rows},
	volume = {86},
	year = {2010}}

@article{faselStablyFreeModules2012,
	author = {Fasel, Jean and Swan, Richard G. and Rao, Ravi A.},
	journal = {Publ. math. IHES},
	month = nov,
	number = {1},
	pages = {223--243},
	title = {On stably free modules over affine algebras},
	volume = {116},
	year = {2012}}

@article{faselSuslinsCancellationConjecture2025,
	author = {Fasel, Jean},
	journal = {Duke Math. J.},
	number = {12},
	pages = {2383--2423},
	title = {Suslin's cancellation conjecture in the smooth case},
	volume = {174},
	year = 2025}

@incollection{faselVasersteinSymbolReal2018,
	address = {New Delhi},
	author = {Fasel, Jean},
	booktitle = {$\mathrm{K}$-{{Theory}}},
	editor = {Srinivas, Vasudevan and Roushon, Sayed K. and Rao, Ravi A. and Parameswaran, A. J. and Krishna, Amalendu},
	number = {19},
	publisher = {Hindustan Book Agency},
	series = {Tata Inst. Fund. Res. Publ.},
	title = {The {{Vaserstein}} symbol on real smooth affine threefolds},
	year = 2018}

@article{grothendieckQuelquesPointsDalgebre1957,
	author = {Grothendieck, Alexander},
	journal = {Tohoku Mathematical Journal},
	number = {2},
	title = {Sur quelques points d'alg{\`e}bre homologique, {{I}}},
	volume = {9},
	year = {1957}}

@book{haesemeyerNormResidueTheorem2019,
	address = {Princeton},
	author = {Haesemeyer, Christian and Weibel, Charles A.},
	number = {375},
	publisher = {Princeton University Press},
	series = {Annals of {{Mathematics Studies}}},
	title = {The {{norm residue theorem}} in {{motivic cohomology}}},
	year = 2019}

@article{hornbostelA1representabilityHermitianKtheory2005,
	author = {Hornbostel, Jens},
	journal = {Topology},
	number = {3},
	pages = {661--687},
	title = {$\mathbb{A}^1$-representability of Hermitian {{$\Kr$-theory}} and {{Witt}} groups},
	volume = {44},
	year = 2005}

@article{jacobsonRealCohomologyPowers2017,
	author = {Jacobson, Jeremy},
	journal = {Ann. $K$-Theory},
	number = {3},
	pages = {357--385},
	title = {Real cohomology and the powers of the fundamental ideal in the {{Witt}} ring},
	volume = {2},
	year = {2017}}

@article{kumarAlgebraicCyclesVector1982,
	author = {Kumar, N. Mohan and Murthy, M. Pavaman},
	journal = {Ann. Math.},
	number = {3},
	pages = {579--591},
	title = {Algebraic cycles and vector bundles over affine three-folds},
	volume = {116},
	year = 1982}

@article{kumarStablyFreeModules1985,
	author = {Kumar, N. Mohan},
	journal = {Am. J. Math.},
	number = {6},
	pages = {1439--1444},
	title = {Stably {{free modules}}},
	volume = {107},
	year = 1985}

@unpublished{lerbetCohomologicalClassificationVector2026,
	author = {Lerbet, Samuel},
	note = {{arXiv:2605.22706}},
	title = {{On the cohomological classification of vector bundles on smooth real affine surfaces and threefolds}},
	year = {2026}}

@article{lerbetImageHigherSignature2026,
	author = {Lerbet, Samuel},
	journal = {Ann. $K$-Theory},
	number = {2},
	pages = {171--212},
	title = {On the image of higher signature maps},
	volume = {11},
	year = 2026}

@book{lurieHigherToposTheory2009,
	address = {Princeton, N.J},
	author = {Lurie, Jacob},
	number = {170},
	publisher = {Princeton university press},
	series = {Annals of Mathematics Studies},
	title = {Higher topos theory},
	year = {2009}}

@book{morelA1AlgebraicTopologyField2012,
	address = {Berlin, Heidelberg},
	author = {Morel, Fabien},
	publisher = {Springer},
	series = {Lecture {{Notes}} in {{Mathematics}}},
	title = {$\mathbb{A}^1$-{{Algebraic Topology}} over a {{Field}}},
	volume = {2052},
	year = {2012}}

@article{morelA1homotopyTheorySchemes1999,
	author = {Morel, Fabien and Voevodsky, Vladimir A.},
	journal = {Publ. math. IHES},
	number = {1},
	pages = {45--143},
	title = {$\mathbb{A}^1$-{h}omotopy {t}heory of {s}chemes},
	volume = {90},
	year = {1999}}

@article{murthyZeroCyclesProjective1994,
	author = {Murthy, M. Pavaman},
	journal = {Ann. of Math.},
	number = {2},
	pages = {405--434},
	title = {Zero cycles and projective modules},
	volume = {140},
	year = {1994}}

@article{nikolausPrincipalinftybundlesGeneral2015,
	author = {Nikolaus, Thomas and Schreiber, Urs and Stevenson, Danny},
	journal = {J. Homotopy Relat. Struct.},
	month = dec,
	number = {4},
	pages = {749--801},
	shorttitle = {Principal $\infty$-Bundles},
	title = {Principal $\infty$-bundles: general theory},
	volume = {10},
	year = {2015}}

@article{rostChowGroupsCoefficients1996,
	author = {Rost, Markus},
	journal = {Documenta Math.},
	pages = {319--393},
	title = {Chow groups with coefficients},
	volume = {1},
	year = {1996}}

@article{schlichtingEulerClassGroups2017,
	author = {Schlichting, Marco},
	journal = {Adv. Math.},
	pages = {1--81},
	title = {Euler class groups and the homology of elementary and special linear groups},
	volume = {320},
	year = {2017}}

@article{schlichtingHermitianKtheoryDerived2017,
	author = {Schlichting, Marco},
	journal = {J. Pure Appl. Algebra},
	number = {7},
	pages = {1729--1844},
	title = {Hermitian {{$K$-theory}}, derived equivalences and {{Karoubi}}'s fundamental theorem},
	volume = {221},
	year = 2017}

@article{suslinCancellationTheoremProjective1977,
	author = {Suslin, Andrei A.},
	journal = {Dokl. Akad. Nauk SSSR},
	number = {4},
	pages = {808--811},
	title = {A cancellation theorem for projective modules over algebras},
	volume = {236},
	year = {1977}}

@article{suslinStablyFreeModules1977,
	author = {Suslin, Andrei A.},
	journal = {Mathematics of the USSR-Sbornik},
	number = {4},
	pages = {479--491},
	title = {On stably free modules},
	volume = {31},
	year = {1977}}

@article{swanCancellationTheoremProjective1974,
	author = {Swan, Richard G.},
	journal = {Invent. math.},
	number = {1},
	pages = {23--43},
	title = {A cancellation theorem for projective modules in the metastable range},
	volume = {27},
	year = {1974}}

@article{voevodskyMotivicCohomology2coefficients2003,
	author = {Voevodsky, Vladimir A.},
	journal = {Publ. math. IHES},
	number = {1},
	pages = {59--104},
	title = {Motivic cohomology with $\mathbb{Z}/2$-coefficients},
	volume = {98},
	year = {2003}}

@book{weibelKbookIntroductionAlgebraic2013,
	address = {Providence, R.I.},
	author = {Weibel, Charles A.},
	number = {vol. 145},
	publisher = {American Mathematical Society},
	series = {Graduate Studies in Mathematics},
	title = {The {{$K$-book}}: an introduction to algebraic {{$K$-theory}}},
	year = {2013}}

@misc{stacks-project,
	author = {The {Stacks project authors}},
	howpublished = {\url{https://stacks.math.columbia.edu}},
	title = {The {Stacks project}},
	year = {2026}}
\end{footnotesize}}

\end{document}